\documentclass[11pt,leqno]{amsart}
\usepackage[latin1]{inputenc}

\usepackage{latexsym, amsmath, amssymb, esint}
\usepackage{bm}
\usepackage{graphics,epsfig,graphicx,graphics}
\usepackage{color, xcolor}
\usepackage{tikz}
\usepackage[hyperpageref]{backref}
\usetikzlibrary{patterns}
\usepackage[english]{babel}
\usepackage[colorlinks=true,urlcolor=blue, citecolor=red,linkcolor=blue,
linktocpage,pdfpagelabels, bookmarksnumbered,bookmarksopen]{hyperref}

\newtheorem{theorem}{Theorem}[section]

\newtheorem{lemma}[theorem]{Lemma}
\newtheorem{proposition}[theorem]{Proposition}

\theoremstyle{definition}
\newtheorem{remark}[theorem]{Remark}

\newcommand{\R}{{\mathbb R \,}}

\newcommand{\QQ}{\mathcal{Q}}

\newcommand{\BHO}{B_{H_0}}

\numberwithin{equation}{section}

\begin{document}
    \title[Stress concentration in nonlinear perfect conductivity problems]{Stress concentration for closely located inclusions in nonlinear perfect conductivity problems}

  \date{}
%
%\author[G. Ciraolo]{Giulio Ciraolo}
%\address{Dipartimento di Matematica e
%Informatica, Universit\`a di Palermo, Via Archirafi 34, 90123, Italy.}
%\email{giulio.ciraolo@unipa.it}
%\urladdr{http://www.math.unipa.it/~g.ciraolo/}
%
%\author[L. Vezzoni]{Luigi Vezzoni}
%\address{...}
%\email{luigi.vezzoni@unito.it}
%\urladdr{http://www.luigione.it}
%
%    \keywords{...}
%    \subjclass{Primary ...; Secondary ...}
%
%
\author{Giulio Ciraolo and Angela Sciammetta}

 \address{Giulio Ciraolo \\ Dipartimento di Matematica e Informatica \\ Universit\`a di Palermo\\ Via Archirafi 34\\ 90123 Palermo\\ Italy}
\email{giulio.ciraolo@unipa.it}

 \address{Angela Sciammetta \\ Dipartimento di Matematica e Informatica \\ Universit\`a di Palermo\\ Via Archirafi 34\\ 90123 Palermo\\ Italy}
\email{angela.sciammetta@unipa.it}

\keywords{Gradient blow-up, Finsler $p$-Laplacian, perfect conductor}
    \subjclass{Primary: 35J25, 35B44, 35B50; Secondary: 35J62, 78A48, 58J60.}

\begin{abstract}
We study the stress concentration, which is the gradient of the solution, when two smooth inclusions are closely located in a possibly anisotropic medium $\Omega \subset \mathbb{R}^N$, $N \geq 2$. The governing equation may be degenerate of $p-$Laplace type, with $1<p \leq N$. We prove optimal $L^\infty$ estimates for the blow-up of the gradient of the solution as the distance between the inclusions tends to zero. 
\end{abstract}
\maketitle

%\tableofcontents

\section{Introduction}
When two inclusions are closely located, it may occur that the stress concentrates in some region and it may cause a failure if any of the principal material strains exceed their respective tensile failure strains. Hence, a theoretical study predicting the possible failure initiation is of great importance for the applications and, in the last two decades, quantitative results for the stress concentration in composite materials have been the goal of many studies.

Our study originates from the paper of Babu\v{s}ka et al. \cite{BabuskaEtc}, where the problem of smooth inclusions closely located in a background linear material was studied numerically. From the mathematical point of view, one may consider a domain $\Omega \subset \mathbb{R}^N$, $N \geq 2$, representing the background matrix, and two inclusions $D_\delta^1,D_\delta^2 \subset \Omega$ which are located at \emph{small} distance $\delta$ and far from the boundary of $\Omega$. The modeling problem is formulated as follows
\begin{equation} \label{pbLinear1}
\begin{cases}
\text{div}\left(a_k(x)\nabla u \right) = 0& \hbox{in } \,  \Omega, \\
u=\varphi & \hbox{on } \, \partial \Omega,
\end{cases}
\end{equation}
where  $\varphi \in C^0(\partial \Omega)$ is a potential prescribed on the boundary of $\Omega$ and
\begin{equation*}
a_k(x)=\left\{
  \begin{array}{ll}
    1, & \hbox{$\Omega \setminus ( D^1_\delta \cup D^2_\delta)$,} \\
    k, & \hbox{$D^1_\delta \cup D^2_\delta$,}
  \end{array}
\right.
\end{equation*}
with $k \in (0,+\infty)$ (see for instance \cite{BaoLiYin}). In \cite{BabuskaEtc} the authors showed numerically that $\|\nabla u_\delta\|_{L^\infty(\Omega)}$ is bounded independently of the distance $\delta$ between $D^1_\delta$ and $D^2_\delta$. Later, Bonnetier and Vogelius \cite{BonnVog} rigorously proved this result for $N=2$ when  $D^1_\delta$ and $D^2_\delta$ are two unit balls, and Li and Vogelius \cite{LiVog} extended the results to general second order elliptic equations with piecewise smooth coefficients. The problem was also studied for general second order elliptic systems by Li and Nirenberg in \cite{LiNir}.

The behavior of the gradient of the solution may be very different when $k$ degenerates to zero or infinity and one may have stress concentration close to the points where the inclusions touch at the limit $\delta=0$. 

In this paper, we are interested in the perfect conductivity case, i.e. when $k=+\infty$. This case is modeled by the problem\begin{equation} \label{perf_cond_isotr}
\begin{cases}
\Delta u = 0 & \text{in } \Omega_\delta \\
|\nabla u|=0 & \text{in } D^i_\delta \,, i=1,2, \\
\displaystyle\int_{\partial D^i_\delta} u_\nu = 0 & i=1,2,  \\
u = \varphi & \text{on } \partial \Omega \,,
\end{cases}
\end{equation}
where $\nu$ denotes the outward normal to $D^i_\delta$, $i=1,2$, and we set 
$$
\Omega_\delta= \Omega \setminus \overline{D_\delta^1 \cup D_\delta^2} 
$$
 (see for instance \cite{BaoLiYin}). In the case of smooth inclusions, it has been proved that the optimal blow-up rate of $|\nabla u|$ is $\delta^{-1/2}$ for $N=2$, it is $(\delta |\log \delta|)^{-1}$ for $N=3$  and $\delta^{-1}$ for $N \geq 4$, see \cite{ACKLY,AKL,Ref3,BaoLiLi,BaoLiYin,BaoLiYin_II,BaoLiYin_III,BaoLiYin_IV,BaoLiYin_V,KLY,Kangy_Yu2017,Kang_Yun2017,LL,Yun,YunII} and references therein. In addition to its mathematical interest, the characterization of the gradient blow-up is relevant for the applications in composite materials. Indeed, numerical simulations related to problem \eqref{perf_cond_isotr} may be difficult to perform due to the presence of stress concentration, and in particular in the choice of the mesh which has to be chosen finer and finer as $\delta$ tends to zero. The quantitative characterizations in the paper mentioned before is helpful in this direction, since one can write the solution as $u_\delta=v_\delta +  w_\delta$ where $v_\delta$ is known and carries all the information regarding the blow-up, and $\nabla w_\delta$ remains uniformly bounded as $\delta$ tends to zero and can be computed numerically.

The study of the gradient blow-up has been recently extended to nonlinear cases. In \cite{GorbNovikov} the authors study perfectly conductivity problems involving the $p$-Laplacian, with $p>N$ (see also \cite{Gorb,Novikov}). Nonlinear conductivities of this type may be found in several applications, and we refer to \cite[Section 1]{GorbNovikov} for more details. The mathematical approach in \cite{GorbNovikov} is purely nonlinear and substantially differs from the ones adopted in the linear case. 

In our recent paper \cite{CiraoloSciammetta}, we studied anisotropic conductivities with anisotropy characterized by a norm 
$H: \xi \mapsto H(\xi)$ with $\xi \in \mathbb{R}^N$. More precisely, we considered the anisotropic perfectly conductivity problem 
\begin{equation*} %\label{DpH}%\tag{$P^H_{\delta}$}
\left\{
   \begin{array}{ll}
    \triangle^H u_{\delta} = 0& \hbox{in} \,\,\, \Omega_{\delta}, \\
 H(\nabla u_{\delta}) = 0 & \hbox{in} \,\,\, \overline{D}_{\delta}^i, \ i=1,2\,, \\ %\overline{B_{\delta,H}^i(r)}, \,\, i=1,2, \\
     \displaystyle\int_{\partial D_\delta^i} H\left(\nabla u_{\delta}\right)\nabla_{\xi}H\left(\nabla u_{\delta}\right)\cdot \nu ds=0 & i=1,2,\\
       u_{\delta}=\varphi(x) & \hbox{on} \,\,\, \partial \Omega \,,
    \end{array}
\right.
\end{equation*}
where $D_\delta^1$ and $D_\delta^2$ are two Wulff shapes of possibly different radii, $\Omega_\delta = \Omega \setminus (\overline{D^1_\delta \cup D^2_\delta})$,  and $\Delta^H_p$ denotes the Finsler $p-$Laplacian
$$
\Delta^H u_\delta = \text{div}\big(H\left(\nabla u_{\delta}\right) \nabla_{\xi} H\left(\nabla u_{\delta}\right) \big) \,.
$$
%(here we use $\nabla_\xi H 
The main results in \cite{CiraoloSciammetta} are optimal estimates for the gradient blow-up. In accordance to the isotropic case, we showed that the rate of blow-up is $\delta^{-1/2}$ for $N=2$, it is $(\delta |\log \delta|)^{-1}$ for $N=3$  and $\delta^{-1}$ for $N \geq 4$, and we were able to detect the leading term (which is responsible of the blow-up) as $\delta$ tends to zero.

The purpose of this paper is to twofold: (i) we study the nonlinear conductivity problem for anisotropic $p$-Laplace type equations for any $1<p\leq N$, therefore in the Euclidean case we extend the results in \cite{GorbNovikov} to the case $1<p \leq N$; (ii) we deal with anisotropic conductivity problems, which may be of degenerate type.  

More precisely, we consider the problem
\begin{equation} \label{DpH}%\tag{$P^H_{\delta}$}
\left\{
   \begin{array}{ll}
    \triangle^H_p u_{\delta} = 0& \hbox{in} \,\,\, \Omega_{\delta}, \\
 H(\nabla u_{\delta}) = 0 & \hbox{in} \,\,\, \overline{D}_{\delta}^i, \ i=1,2\,, \\ %\overline{B_{\delta,H}^i(r)}, \,\, i=1,2, \\
     \displaystyle\int_{\partial D_\delta^i} H^{p-1}\left(\nabla u_{\delta}\right)\nabla_{\xi}H\left(\nabla u_{\delta}\right)\cdot \nu ds=0 & i=1,2,\\
       u_{\delta}=\varphi(x) & \hbox{on} \,\,\, \partial \Omega \,,
    \end{array}
\right.
\end{equation}
where $D_\delta^1$ and $D_\delta^2$ are two Wulff shapes of possibly different radii $R_1$ and $R_2$, respectively, $\Omega_\delta = \Omega \setminus (\overline{D^1_\delta \cup D^2_\delta})$, $\nu$ is the outward normal to $\partial D_\delta^i$, and $\Delta^H_p$ denotes the Finsler $p-$Laplacian
$$
\Delta^H_p u_\delta = \text{div}\big(H^{p-1}\left(\nabla u_{\delta}\right) \nabla_{\xi} H\left(\nabla u_{\delta}\right) \big) \,,
$$
which has to be understood in the weak sense
$$
\int_{\Omega_\delta} H^{p-1}\left(\nabla u_{\delta}\right) \nabla_{\xi} H\left(\nabla u_{\delta}\right)  \cdot \nabla \phi \, dx = 0 \qquad \text{ for any } \phi \in C_0^1(\Omega_\delta) \,.
$$
Problem \eqref{DpH} can be seen as the Euler-Lagrange equation of the variational problem 
\begin{equation}\label{energia}
\min_{v \in W^{1,p}_{\varphi}(\Omega)} \left\{ \frac1p  \int_\Omega H(\nabla v)^p dx \, :  \ H(\nabla v)=0 \ \text{ in } D^i_\delta \,, \ i=1,2  \right\} \,,
\end{equation}
where
$$
W^{1,p}_\varphi (\Omega) = \left\{ v \in W^{1,p}(\Omega) \, :\ v = \phi \text{ on } \partial \Omega \right\} \,.
$$

We assume that the anisotropic distance of the inclusions from the boundary of $\Omega$ is uniformly bounded by below, i.e.
\begin{equation} \label{dist_B1B2_K}
  \text{dist}_{H_0}\left(\partial \Omega, D^1_{\delta} \cup D^2_{\delta} \right)\geq K,
\end{equation}
for some fixed $K>0$ and that the distance between the two inclusions is very small, i.e.
\begin{equation*}
  \text{dist}_{H_0}\left( D^1_{\delta},  D^2_{\delta}\right) = \delta,
\end{equation*}
for some $0 < \delta \leq \delta_0$. Here, $\text{dist}_{H_0}$ denotes the distance in the ambient norm $H_0$ (which is the dual norm of $H$).

We emphasize that the solution $u_{\delta}$ to \eqref{DpH} is constant on each particle $D_{\delta}^i$ with $i=1,2$, i.e.
\begin{equation}\label{potenziale}
  u_{\delta}=\mathcal{U}^i_{\delta} \quad  \text{ on } D^i_{\delta} \,,
\end{equation}
with $\mathcal{U}^i_{\delta} \in \mathbb{R}$, $i=1,2$, and the values $\mathcal{U}^1_{\delta}$ and  $\mathcal{U}^2_{\delta} $ are unknowns of the problem and have to be determined by solving the minimization problem \eqref{energia}. As we will show, the difference of potentials $\mathcal{U}^1_{\delta} - \mathcal{U}^2_{\delta} $ is responsible of the blow-up of the gradient of the solution as $\delta \to 0^+$.

We are going to describe the limit behavior of the solution in terms of the solution of the problem corresponding to $\delta=0$ (when the two inclusions touch each other), which is given by
\begin{equation} \label{PH0}%\tag{$P^H_0$}
\left\{
   \begin{array}{ll}
    \triangle^H_p u_{0} = 0& \hbox{in} \,\,\, \Omega_{0}, \\
 H(\nabla u_{0}) = 0 & \hbox{in} \,\,\, \overline{D_{0}^i}, \,\, i=1,2, \\
     \displaystyle \sum_{i=1,2} \int_{\partial D_0^i} H^{p-1}\left(\nabla u_{0}\right)\nabla_{\xi}H\left(\nabla u_{0}\right)\cdot \nu ds=0  \,,& \\
       u_{0}=\varphi(x) & \hbox{on} \,\,\, \partial \Omega \,.
    \end{array}
\right.
\end{equation}
A remarkable point is the fact that the solution $u_\delta$ does not converge to $u_0$ in the whole $\Omega_0$ (it is not difficult to show that the gradient of $u_0$ is uniformly bounded). Instead, the convergence in $C^{1,\alpha}$-norm holds true in  compact sets of $\Omega_0$ not including the touching point between $D_0^1$ and $D_0^2$ (which are the two inclusions at the limit $\delta=0$). The behavior of $u_\delta$ close to the limit touching point is described in terms of the following quantity related to $u_0$:
\begin{equation} \label{R0_def}
\mathcal{R}_0 = \int_{\partial D_0^1} H^{p-1}\left(\nabla u_{0}\right)\nabla_{\xi}H\left(\nabla u_{0}\right)\cdot \nu ds \,.
\end{equation}
We emphasize that $\mathcal{R}_0$ is one of the two addends appearing in the third condition of \eqref{PH0}, and  we notice that the third condition in \eqref{PH0} is different from third condition in \eqref{DpH}, since in \eqref{PH0} it is required that the sum of the two integrals on $\partial D_0^1$ and $\partial D_0^2$ vanishes.

Without loss of generality, we  assume that the two inclusions $\partial D_\delta^1$ and $\partial D_\delta^2$ move along  
the $x_N$-axis as $\delta \to 0^+$. Since the inclusions are Wulff shapes, the limit-touching point on $\partial D_\delta^1$ is a point of the form $R_{1} \hat P$, where $\hat P=(0, \ldots, 0, t_0)  \in \partial B_{H_{0}}(0,1)$. The matrix $\nabla^2H_0(\hat P)$ is crucial to describe the blow-up. More precisely, we denote by $\QQ$ the matrix obtained  by considering the first $N-1$ rows and $N-1$ columns of $\nabla^2H_0(\hat P)$, i.e.
\begin{equation}\label{matriceQ}
\QQ=\left(
\begin{array}{cccc}
      \partial^2_{x_1x_1}H_0(\hat P) & \ldots & \partial^2_{x_1 x_{N-1}}H_0(\hat P) \\
         \vdots & \ddots & \vdots \\
      \partial^2_{x_{N-1}x_1}H_0(\hat P) & \ldots & \partial^2_{x_{N-1}x_{N-1}}H_0(\hat P) \\
    \end{array}
  \right) \,,
\end{equation}
(we shall use the variable $x$ for the ambient space $\mathbb{R}^N$ and $\xi$ for the dual space). We also recall the definition of \emph{anisotropic normal} $\nu_H$ at a point $x$, which is given by
$$
\nu_H (x) = \nabla_{\xi}H\left(\nu(x)\right) \,,
$$
where $\nu(x)$ denotes the outward Euclidean normal at $x$.
Our main result is the following.

\begin{theorem} \label{thm_main_1}
Let $u_{\delta}$ be the solution to \eqref{DpH} and let $\mathcal{R}_0$ be given by \eqref{R0_def}. For any fixed $\tau \in (0,1/2]$ we have
\begin{equation} \label{estimates_main_thm}
  (1-\tau) C_*\Phi_N(\delta)+o\left(\Phi_N(\delta)\right)\leq \|  H(\nabla u_{\delta}) \|^{p-1}_{L^{\infty}(\Omega_{\delta})} \leq (1+\tau) C_*\Phi_N(\delta)+o\left(\Phi_N(\delta)\right)
\end{equation}
as $\delta \to 0^+$, with
$$
\Phi_N(\delta)=\begin{cases}
\delta^{-\frac{N-1}{2}} & \dfrac{N+1}{2}<p \leq N \,, \\
\dfrac{1}{\delta^{p-1}|\ln\delta|} & p=\dfrac{N+1}{2}\,, \\
\dfrac{1}{\delta^{p-1}} & 1< p<\dfrac{N+1}{2} \,,
\end{cases}
$$
and
\begin{equation} \label{C*}
C_* = \left(\dfrac{R_1+R_2}{2R_1R_2}\right)^{\frac{N-1}{2}}|\mathcal{Q}|^{\frac{N-1}{2}}\mathcal{R}_0 C \,,
\end{equation}
where $Q$ is given by \eqref{matriceQ} and $C$ depends on $N$ and $\nu_{H}(\hat P) \cdot \nu(\hat P)$.
\end{theorem}

Theorem \ref{thm_main_1} gives an optimal quantitative description of the blow-up of the gradient for problem \eqref{DpH}.
Moreover, the estimates \eqref{estimates_main_thm} provide an \emph{almost} sharp characterization of the leading term in the blow-up. Indeed, $\tau$ may be chosen as small as desired which suggests that $u_\delta \sim  C_*\Phi_N(\delta)$ as $\delta \to 0^+$.

Compared to \cite{GorbNovikov} we deal with nonlinear problems of $p$-Laplace type for any  $1<p \leq N$. Theorem \ref{thm_main_1} is the natural extension to the anisotropic $p$-Laplace conductivity problems studied in \cite{CiraoloSciammetta}.
The set-up of the proof of Theorem \ref{thm_main_1} differs from the classical ones used in the linear cases and it is in the spirit of the ones adopted in \cite{CiraoloSciammetta} and \cite{GorbNovikov}. More precisely, we define a \emph{neck} of width $w>0$ (and sufficiently small) as the set
\begin{equation}\label{neck}
\mathcal{N}_{\delta}(w)=\{x=(x',x_N) \in \Omega_{\delta}\,\, \text{such that} \,\, |\QQ^{\frac{1}{2}}x'|<w, H_{0}(x)<\max(R_1,R_2) \},
\end{equation}
(see Figure \ref{fig_neck}) where $\QQ^{\frac{1}{2}}$ is the square root of the matrix $\QQ$ defined in \eqref{matriceQ}. We first show that  $\nabla u_\delta$ remains uniformly bounded outside the neck as $\delta \to 0^+$: this is achieved by using comparison principles and employing some maximum principles for a suitable $P$-function. Beyond the degeneracy of the operator, this is one of the major points where the extension to the degenerate case requires new tools (see Remark \ref{remark_Pfunction} below). Then we focus on what happens inside the neck, and we give sharp estimates on the difference of potential $\mathcal{U}_\delta^1-\mathcal{U}_\delta^2$ as $\delta \to 0^+$, which leads to \eqref{estimates_main_thm}.

\begin{center}
	\begin{figure}[h]
		\includegraphics[scale=0.5]{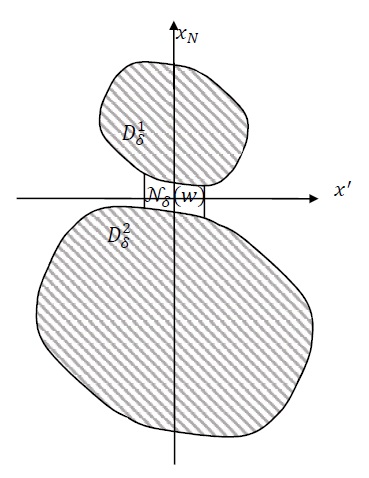}\label{fig_neck}
		\caption{Two anisotropic balls and the neck between them.}
	\end{figure}
\end{center}

The paper is organized as follows. In Section \ref{Basic notations and preliminary results} we give some preliminary results and set up the notation. In Section \ref{sect_max_princ} we prove some crucial maximum principles and in Section \ref{sect_unif_bounds} we give uniform estimates for the gradient in the region where it remains uniformly bounded. Section \ref{sect_thm1} is devoted to the proof of Theorem \ref{thm_main_1}.

\section{Preliminaries}\label{Basic notations and preliminary results}
\subsection{About norms in $\mathbb{R}^N$} In this subsection we recall some facts about norms in $\mathbb{R}^N$, $N \geq 2$. Let $H:\mathbb{R}^N \rightarrow \mathbb{R}$ be a norm, i.e.
\begin{itemize}
 \item[(i)]  $H$ {is convex,}    
 \item[(ii)]  $H(\xi)\geq 0 \,\, \text{for} \,\,\xi \in \mathbb{R}^N$ {and}   $H(\xi)=0$ {if and only if} $\xi=0$,   
 \item[(iii)] $H(t\xi)=|t|H(\xi)$  {for} $\xi \in \mathbb{R}^N$ {and} $t \in \mathbb{R}$. 
\end{itemize}

\noindent Since all norms in $\mathbb{R}^N$ are equivalent, there exist two positive constants $c_1,c_2$ such that
$$
c_1 |\xi| \leq H(\xi) \leq c_2 |\xi| \quad \text{ for any } \xi \in \mathbb{R}^N \,.
$$
In our notation, $H_0$ is a norm in the ambient space, and the dual space (still identified with $\mathbb{R}^N$) is equipped by the dual norm $H$, therefore
\begin{equation*} 
  H_0(x)=\sup_{\xi \neq 0} \dfrac{x \cdot \xi}{H(\xi)} \quad \text{for} \quad x\in \mathbb{R}^N \,;
\end{equation*}
and analogously
\begin{equation*} 
  H(\xi)=\sup_{x \neq 0} \dfrac{x \cdot \xi}{H_0(\xi)} \quad \text{for} \quad x \in \mathbb{R}^N \,.
\end{equation*}
Hence, according to this notation, the norm of the gradient of a function $u$ will be given by using $H$.

By assuming that $H$ is smooth enough outside the origin, the homogeneity of $H$ yields
\begin{equation} \label{nabla_H_zero_om}
  \nabla_{\xi}H(t \xi)= \text{sign}(t)  \nabla_{\xi}H(\xi), \quad \text{for} \quad \xi \neq 0 \quad \text{and} \quad t\neq 0,
\end{equation}
and\begin{equation}\label{condizionescalare}
 \nabla_{\xi}H(\xi) \cdot \xi  =H(\xi), \quad \text{for} \quad \xi \in \mathbb{R}^N,
\end{equation}
where the left hand side is taken to be $0$ when $\xi=0$. Moreover,
\begin{equation}\label{nabla2_H_omog}
  \nabla^2_{\xi}H(t\xi)=\dfrac{1}{|t|}\nabla^2_{\xi}H(\xi), \quad \text{for} \quad \xi \neq 0 \quad \text{and} \quad t\neq 0 \,,
\end{equation}
where $\nabla^2_{\xi}$ is the Hessian operator with respect to the $\xi$ variable; we also notice that
\begin{equation}\label{nabla2_H2_omog}
  \nabla^2_{\xi}H^2(t\xi)=\nabla^2_{\xi}H^2(\xi), \quad \text{for} \quad \xi \neq 0 \quad \text{and} \quad t\neq 0 \,.
\end{equation}
Hence, \eqref{condizionescalare} implies that
\begin{equation}\label{H=0}
   \partial^2_{\xi_{i}\xi_{k}}H (\xi) \xi_{i}=0,
\end{equation}
with $\xi \neq 0$ and for every $k=1, \ldots,N$. Moreover, by differentiating \eqref{H=0} we obtain that 
\begin{equation} \label{HH=0}
  \partial^3_{\xi_{i} \xi_j \xi_{k}}H (\xi) \xi_{i} + \partial^2_{\xi_{j}\xi_{k}}H (\xi) =0 
\end{equation}
for $\xi \neq 0$.

 The following properties hold provided that $H \in C^1\left(\mathbb{R}^N\setminus\{0\}\right)$ and the unitary ball $\{\xi \in \mathbb{R}^n:\ H(\xi) < 1 \}$ is strictly convex (see \cite[Lemma 3.1]{CianchiSalani}):
\begin{equation}\label{LemmaCianchiSalani1}
 H_0\left(\nabla_{\xi}H(\xi)\right)=1, \quad \text{for} \quad \xi \in \mathbb{R}^N\setminus\{0\},
\end{equation}
and
\begin{equation}\label{LemmaCianchiSalani11}
  H\left(\nabla H_0(x)\right)=1, \quad \text{for} \quad x \in \mathbb{R}^N\setminus\{0\} \,.
\end{equation}

For $\xi_0 \in \mathbb{R}^N$ and $r>0$, the ball of center $\xi_0$ and radius $r$ in the norm $H$ is denoted by
\begin{equation*}
B_H(\xi_0,r)=\{\xi \in \mathbb{R}^N :H(\xi-\xi_0)<r\};
\end{equation*}
analgously,
\begin{equation*}
  B_{H_{0}}(x_0,r)=\{x \in \mathbb{R}^N :H_0(x-x_0)<r\}
\end{equation*}
denotes the ball of center $x_0$ and radius $r$ in the norm $H_0$. A ball in the norm $H_0$ is called the \emph{Wulff shape} of $H$.

Let $\BHO(r)$ and $\BHO(R)$ be two Wulff shapes centered at the origin, with $r<R$. It will be useful to have at hand the explicit solution to the problem
\begin{equation} \label{pb_v_anello}
\left\{
   \begin{array}{ll}
    \Delta^H_p v = 0& \hbox{in} \,\,\, \BHO(R) \setminus \overline{\BHO(r)}, \\
       v=C_r & \hbox{on} \,\,\, \partial \BHO(r),\\
v=C_R & \hbox{on} \,\,\, \partial \BHO(R),
    \end{array}
\right.
\end{equation}
which is given by
\begin{equation} \label{v_anello}
v(x)=\left\{
   \begin{array}{ll}
    (C_r-C_R)\dfrac{H_0(x)^{\frac{p-N}{p-1}}-R^{\frac{p-N}{p-1}}}{r^{\frac{p-N}{p-1}}-R^{\frac{p-N}{p-1}}}+C_R& \hbox{if} \,\,\, 1<p<N, \\
\\
      (C_r-C_R)\dfrac{\ln \left(R^{-1} H_0(x)\right)}{\ln \left(R^{-1}r \right)}+C_R& \hbox{if} \,\,\, N=p, \\
    \end{array}
\right.
\end{equation}
for any $x \in \overline{\BHO(R)} \setminus \BHO(r)$.

\subsection{Existence and uniqueness}
As mentioned in the Introduction, we consider the perfectly conductivity problem \eqref{DpH}, which is the Euler-Lagrange equation for the variational problem \eqref{energia}.  It is well-known that $u \in C^{1,\alpha}(\Omega_\delta)$ (see \cite{Dibenedetto}). In the following we prove the existence and uniqueness of solution.

\begin{theorem}
  There exists at most one solution $u \in H^1(\Omega_\delta) \cap C^{1,\alpha}(\overline{\Omega}_{\delta})$ of problem \eqref{DpH}.
\end{theorem}

\begin{proof}
  Let  $u_1$, $u_2 \in H^1(\Omega_\delta)$ be two solutions of \eqref{DpH}. By multiplying the first equation of \eqref{DpH} by $u_1-u_2$ and integrating by parts, for $j \in \{1,2\}$, we have

\begin{eqnarray*}
  0 &=& \displaystyle\int_{\Omega_\delta} H^{p-1}\left(\nabla u_j\right)\nabla_{\xi}H\left(\nabla u_j\right)\cdot \nabla (u_1-u_2) dx - \displaystyle\int_{\partial \Omega} H^{p-1}(\nabla u_j)\nabla_{\xi} H(\nabla u_j) (u_1-u_2) \cdot \nu ds \nonumber \\
& & \quad +  \sum_{i=1}^2 \displaystyle\int_{\partial D_\delta^i} H^{p-1}(\nabla u_j)\nabla_{\xi} H(\nabla u_j) (u_1-u_2) \cdot \nu ds\nonumber  \\
&=&  \displaystyle\int_{\Omega_\delta} H^{p-1}\left(\nabla u_j\right)\nabla_{\xi}H\left(\nabla u_j\right)\cdot \nabla (u_1-u_2) dx  \,,\nonumber  \\
\end{eqnarray*}
where in the last equality we used the fourth condition in \eqref{DpH} and the fact that $u_1=u_2$ on $\partial \Omega$.
Thus, by the strong convexity of $H$, we have
\begin{equation*}
  0=\displaystyle\int_{\Omega_\delta} \left( H^{p-1}\left(\nabla u_1\right)\nabla_{\xi}H\left(\nabla u_1\right)-H^{p-1}\left(\nabla u_2\right)\nabla_{\xi}H\left(\nabla u_2\right)\right)\cdot \nabla (u_1-u_2) dx \geq \lambda \displaystyle\int_{\Omega_\delta} \left|\nabla (u_1-u_2)\right|^p dx \geq 0.
\end{equation*}
Thus $\nabla u_1= \nabla u_2$ in $\Omega_\delta$ and, since $u_1=u_2$ on $\partial D_\delta^i$, we have $u_1=u_2$ in $\Omega_\delta$.
\end{proof}

We define the energy functional
\begin{equation*}
  I_{\infty}[u]=\dfrac{1}{p} \displaystyle\int_{\Omega_{\delta}} H^p\left(\nabla u\right) dx,
\end{equation*}
where $u$ belongs to the set
\begin{equation*}
  \mathcal{A} := \left\{u \in W^{1,p}_{\varphi}(\Omega) : H\left(\nabla u\right) =0 \,\, \text{on} \,\, \overline{D_\delta^1 \cup D_\delta^2}\right\}.
\end{equation*}

\begin{theorem}
There exists a minimizer $u \in\mathcal{A}$ satisfying
\begin{equation*}
   I_{\infty}[u]=\min_{v \in \mathcal{A}} I_{\infty}[v].
\end{equation*}
 Moreover, $u \in W^{1,p}(\Omega_\delta) \cap C^{1,\alpha}(\overline{\Omega}_{\delta})$ is a solution to \eqref{DpH}.
\end{theorem}

\begin{proof}
The existence of the minimizer and the Euler Lagrange equation $\Delta_p^H u = 0$ follow from standard methods in the calculus of variations.
 The only thing which we need to show is the fourth equation in \eqref{DpH}. Let $i , j=1,2$ be fixed and let $\phi \in C^{\infty}_0(\Omega)$ be such that
\begin{equation*}
  \phi = \left\{
              \begin{array}{lllll}
                1, & \hbox{on}& \partial D_\delta^i, & &\\
                0, & \hbox{on}& \partial D_\delta^j, & \hbox{for}& j\neq i.
              \end{array}
            \right.
\end{equation*}
Since $u$ is a minimizer, by integrating by parts we obtain
\begin{eqnarray*}
  0 &=& - \displaystyle\int_{\Omega_\delta} \text{div}\left(H^{p-1}(\nabla u)\nabla_{\xi} H(\nabla u)\right) \phi \, dx \nonumber \\
 &=& \displaystyle\int_{\Omega_\delta} H^{p-1}\left(\nabla u\right)\nabla_{\xi}H\left(\nabla u\right)\cdot\nabla \phi \, dx - \displaystyle\int_{\partial \Omega} H^{p-1}(\nabla u)\nabla_{\xi} H(\nabla u) \phi \cdot \nu ds \nonumber \\
& & \quad + \sum_{j=1}^2 \displaystyle\int_{\partial D_\delta^j} H^{p-1}(\nabla u)\nabla_{\xi} H(\nabla u) \phi \cdot \nu ds\nonumber  \\
&=&  \displaystyle\int_{\partial D_\delta^i} H^{p-1}\left(\nabla u\right)\nabla_{\xi}H\left(\nabla u\right)\cdot \nabla \phi \,dx \nonumber  \\
\end{eqnarray*}
and we conclude.
\end{proof}

\section{Maximum principles} \label{sect_max_princ}
In this section we prove some maximum principles for $u_\delta$, $H(\nabla u)$ and for a $P$-function which is suitable for our purposes.

  We first recall that the Finsler $p-$Laplacian fulfills the maximum and comparison principles (see for instance \cite[Lemma 2.3]{BC}). In the following lemma, we show that the maximum and minimum of $u_\delta$ are attained at the boundary of $\Omega$ (and not on $\partial D_\delta^i$, $i=1,2$).
\begin{lemma} \label{lemma_maxmin}
Let $u_{\delta}$ the solution to problem \eqref{DpH}. The maximum and the minimum of $u_{\delta}$ are attained on $\partial \Omega$. In particular, we have that
$$
\max_{\overline{\Omega}_\delta} |u_\delta| =  \max_{\partial \Omega} |\varphi| \,.
$$
\end{lemma}
\begin{proof}
From the maximum principle (see for instance \cite[Lemma 2.3]{BC}) we have that $|u_\delta|$ attains its maximum on $\partial \Omega_\delta$. By contradiction, let assume that $\max u_{\delta}=\mathcal{U}_\delta^1$. From Hopf's boundary point lemma we have that $|\nabla u_{\delta}|>0$ on $\partial D_\delta^1$, which contradicts the third condition of \eqref{DpH}. Analogously, the maximum can not be attained at $\partial D_\delta^2$, which implies the assertion.
\end{proof}
Before giving other maximum principles, we set some notation and prove some basic inequalities for the Finsler $p-$Laplacian.
In order to avoid heavy formulas, we use the following notation:
$$
u_i=\dfrac{\partial }{\partial x_i}u(x) \,,\quad  u_{ij}=\dfrac{\partial^2 }{\partial x_i \partial x_j} u(x) \,,\quad
  \partial_{\xi_i}H=\dfrac{\partial }{\partial \xi_i} H(\xi), \quad \partial^2_{\xi_i \xi_j}H=\dfrac{\partial^2 }{\partial \xi_i \partial \xi_j} H(\xi)$$
(we recall that we are going to use the variable $x \in \mathbb{R}^N$ for the ambient space and the variable $\xi \in \mathbb{R}^N$ for the dual space). Since
\begin{equation*}
  \text{div} \left(H^{p-1}(\nabla u)\nabla_{\xi} H(\nabla u)\right) = \left[(p-1) H^{p-2}(\nabla u)\partial_{\xi_{i}}H(\nabla u)\partial_{\xi_{j}}H(\nabla u)+ H^{p-1}(\nabla u)\partial_{\xi_{i} \xi_{j}}^2H(\nabla u)\right]u_{ij}
\end{equation*}
at points where $\nabla u \neq 0$, by setting
\begin{equation}\label{aij}
  a_{ij}:= (p-1) H^{p-2}(\nabla u)\partial_{\xi_{i}}H(\nabla u)\partial_{\xi_{j}}H(\nabla u)+ H^{p-1}(\nabla u)\partial_{\xi_{i} \xi_{j}}^2H(\nabla u)= \dfrac{1}{p} \partial_{\xi_{i} \xi_{j}}^2 H^p(\nabla u) \,,
\end{equation}
the Finsler $p-$Laplacian can be written as
\begin{equation}\label{finsleraijuj}
  \Delta^H_p u= a_{ij}u_{ij} %= {\rm tr}(\nabla^2_{\xi}V(\nabla u) \nabla^2 u)
\end{equation}
at points where $\nabla u \neq 0$. 

In the rest of this section, we shall give some maximum principles involving the second order elliptic operator $\mathcal L$ defined by
\begin{equation}\label{OperatoreL}
  \mathcal{L} v := \partial_i (a_{ij} v_j)  \,.
\end{equation}
%where
%\begin{multline*}
%a_{ijl} u_l   = \partial_{\xi_{j}}H(\nabla u) \partial^2_{{\xi_{i}}{\xi_{l}}}H(\nabla u)u_l+
%\partial_{\xi_{i}}H(\nabla u)\partial^2_{{\xi_{j}}{\xi_{l}}}H(\nabla u)u_l  \\
%+ \partial^2_{{\xi_{i}}{\xi_{j}}}H(\nabla u) \partial_{\xi_{l}}H(\nabla u)u_l + H(\nabla u)\partial^3_{{\xi_{i}}{\xi_{j}}{\xi_{l}}} H(\nabla u) u_l \,.
%\end{multline*}
\begin{lemma} \label{lemma_maxprinc1}
Let $u$ satisfy $\Delta_p^H u = 0$ in some domain $E$, and assume that $\nabla u \neq 0$. Then
\begin{equation}\label{P_2}
  \mathcal{L}u^2=2(p-1)H^{p}(\nabla u) \,.
\end{equation}
\end{lemma}
\begin{proof}
%We notice that ????? we have
%\begin{equation}\label{port1}
%\partial^3_{\xi_i \xi_j \xi_k}H(\nabla u)u_j = -\partial^2_{\xi_i \xi_k}H(\nabla u) 
%\end{equation}.
From \eqref{condizionescalare}, \eqref{H=0} and \eqref{HH=0} we have
\begin{eqnarray*}
\partial_i(a_{ij}) u_j&=&  (p-1)(p-2)H^{p-3}(\nabla u)\partial_{\xi_k}H(\nabla u)\partial_{\xi_i}H(\nabla u) \underbrace{\partial_{\xi_j}H(\nabla u) u_j}_{=H(\nabla u)}u_{ik}\\
&&\quad +(p-1)H^{p-2}(\nabla u)\partial^2_{\xi_i\xi_k}H(\nabla u) \underbrace{\partial_{\xi_j}H(\nabla u) u_j}_{=H(\nabla u)}u_{ik}\\
&&\quad +(p-1)H^{p-2}(\nabla u)\partial_{\xi_i}H(\nabla u)\underbrace{\partial^2_{\xi_j\xi_k}H(\nabla u) u_j}_{=0}u_{ik}\\
&&\quad +(p-1)H^{p-2}(\nabla u)\partial_{\xi_k}H(\nabla u)\underbrace{\partial^2_{\xi_i\xi_j}H(\nabla u) u_j}_{=0}u_{ik}+H^{p-1}(\nabla u)\underbrace{\partial^3_{\xi_i \xi_j \xi_k}H(\nabla u)u_j}_{=-\partial^2_{\xi_i \xi_k}H(\nabla u)}u_{ik} \,,
\end{eqnarray*}
i.e.
\begin{equation*}
%\begin{split}
\partial_i(a_{ij}) u_j  = (p-2) \left[(p-1) H^{p-2}(\nabla u)\partial_{\xi_k}H(\nabla u)\partial_{\xi_i}H(\nabla u) 
+  H^{p-1}(\nabla u)\partial^2_{\xi_i\xi_k}H(\nabla u) \right]  u_{ik}%\end{split}
\end{equation*}
and from \eqref{aij} we obtain that
\begin{equation} \label{fine}
\partial_i(a_{ij}) u_j  =(p-2) a_{ik}u_{ik} = 0\,,
\end{equation}
where the last equality follows from \eqref{finsleraijuj}.
Since 
$$
\text{div} \left(a_{ij} \nabla u^2\right) = 2a_{ij} u_i u_j + 2u a_{ij} u_{ij} +2u \partial_i(a_{ij}) u_j \,,
$$
from \eqref{finsleraijuj} and  \eqref{fine} we have
$$
\text{div} \left(a_{ij} \nabla u^2\right) = 2a_{ij} u_i u_j  
$$
and  \eqref{aij} yields 
$$
\text{div} \left(a_{ij} \nabla u^2\right) = 2(p-1)H^{p}(\nabla u) \,,
$$
which is \eqref{P_2}.
\end{proof}

The following lemma will be useful to find a lower bound for $\mathcal{L}H(\nabla u)^2$.

\begin{lemma}
Let $u$ be a smooth function. Then we have 
%    \begin{eqnarray}\label{P_11}
%      a_{ij}\partial^2_{\xi_k \xi_l}H^2(\nabla u)u_{ik}u_{jl}&\geq& \frac{a_{ij}u_{ij}}{N} \left(\partial_{\xi_i} H(\nabla u) \partial_{\xi_j} H(\nabla u)u_{ij}+H(\nabla u)\partial^2_{\xi_i \xi_j}H(\nabla u)u_{ij}\right) \nonumber \\
%      &&\quad +\frac{N}{N-1}\Big[\dfrac{1}{N}H^{p-1}(\nabla u)\partial^2_{\xi_i \xi_j}H(\nabla u)u_{ij}\\
%      &&\quad -\frac{N-1}{N}(p-1)H^{p-2}(\nabla u)\partial_{\xi_i} H(\nabla u) \partial_{\xi_j} H(\nabla u)u_{ij}\Big]\Big[\dfrac{1}{N}H(\nabla u)\partial^2_{\xi_i \xi_j}H(\nabla u)u_{ij} \nonumber \\
%      &&\quad -\frac{N-1}{N}\partial_{\xi_i} H(\nabla u) \partial_{\xi_j} H(\nabla u)u_{ij}\Big] \nonumber
%       \end{eqnarray}
%       
%      
    \begin{equation}\label{P_11}
      a_{ij}\partial^2_{\xi_k \xi_l}H^2(\nabla u)u_{ik}u_{jl} \geq \frac{a_{ij}u_{ij}}{N} \mathcal{A} +\frac{N}{N-1}  \mathcal{B}  \mathcal{C} \,,
      \end{equation}
where
$$
  \mathcal{A} =  \partial_{\xi_i} H(\nabla u) \partial_{\xi_j} H(\nabla u)u_{ij}+H(\nabla u)\partial^2_{\xi_i \xi_j}H(\nabla u)u_{ij}\,,
$$
$$
  \mathcal{B} = \dfrac{1}{N}H^{p-1}(\nabla u)\partial^2_{\xi_i \xi_j}H(\nabla u)u_{ij}
  -\frac{N-1}{N}(p-1)H^{p-2}(\nabla u)\partial_{\xi_i} H(\nabla u) \partial_{\xi_j} H(\nabla u)u_{ij} \,,
$$
and 
$$
  \mathcal{C} = \dfrac{1}{N}H(\nabla u)\partial^2_{\xi_i \xi_j}H(\nabla u)u_{ij}  -\frac{N-1}{N}\partial_{\xi_i} H(\nabla u) \partial_{\xi_j} H(\nabla u)u_{ij} \,.
$$    
\end{lemma}

\begin{proof}
Let
\begin{equation*}
  A_1=(p-1)H^{p-2}(\nabla u)\partial_{\xi_{i}} H(\nabla u)\partial_{\xi_{j}} H(\nabla u)u_{ij}, \quad B_1=H^{p-1}(\nabla u)\partial^2_{\xi_{i}\xi_{j}} H(\nabla u)u_{ij}\,,
\end{equation*}
\begin{equation*}
 A_2=\partial_{\xi_{i}} H(\nabla u)\partial_{\xi_{j}} H(\nabla u)u_{ij}, \quad  B_2=H(\nabla u)\partial^2_{\xi_{i}\xi_{j}} H(\nabla u)u_{ij}\,.
\end{equation*}
In terms of this notation we have that
\begin{equation*}
  a_{ij}u_{ij}= A_1 + B_1, \quad \text{and} \quad \dfrac{1}{2}\partial^2_{\xi_i\xi_j} H^2(\nabla u)u_{ij}=A_2 + B_2 \,,
\end{equation*}
which implies that the right-hand side of \eqref{P_11} can be written as
$$
 \frac{a_{ij}u_{ij}}{N} \mathcal{A} +\frac{N}{N-1}  \mathcal{B}  \mathcal{C} 
 = \dfrac{(A_1+B_1)(A_2+B_2)}{N}+\dfrac{N}{N-1}\left(\dfrac{B_1}{N}-\dfrac{N-1}{N}A_1\right)\left(\dfrac{B_2}{N}-\dfrac{N-1}{N}A_2\right)
$$
and after some computation we obtain that 
\begin{equation} \label{right}
 \frac{a_{ij}u_{ij}}{N} \mathcal{A} +\frac{N}{N-1}  \mathcal{B}  \mathcal{C} 
 = \  A_1A_2+\dfrac{B_1B_2}{N-1} \,.
\end{equation}

%
%\begin{eqnarray}\label{right}
%\dfrac{(A_1+B_1)(A_2+B_2)}{N}&+&\dfrac{N}{N-1}\left(\dfrac{B_1}{N}-\dfrac{N-1}{N}A_1\right)\left(\dfrac{B_2}{N}-\dfrac{N-1}{N}A_2\right)\nonumber \\
%&=&\dfrac{A_1A_2}{N}+\dfrac{B_1B_2}{N}+\dfrac{B_1B_2}{N(N-1)}+\dfrac{N-1}{N}{A_1A_2}\\
%&=&A_1A_2+\dfrac{B_1B_2}{N-1}\nonumber .
%\end{eqnarray}
\noindent The left hand side of \eqref{P_11} is
\begin{multline}\label{left}
  a_{ij}\partial^2_{\xi_k \xi_l}H^2(\nabla u)u_{ik}u_{jl}=A_1A_2+pH^{p-1}(\nabla u)\partial_{\xi_{i}} H(\nabla u)\partial_{\xi_{j}} H(\nabla u)\partial^2_{\xi_l \xi_k}H(\nabla u)u_{ik}u_{jl} \\
  + H^{p}(\nabla u)\partial^2_{\xi_i \xi_j}H(\nabla u)\partial^2_{\xi_l \xi_k}H(\nabla u)u_{ik}u_{jl}.
\end{multline}
We observe that
\begin{equation} \label{manca1}
  A_1A_2=(p-1)H^{p-2}(\nabla u)\left(\partial_{\xi_{i}} H(\nabla u)\partial_{\xi_{k}} H(\nabla u)u_{ik}\right)^2 \,.
\end{equation}
Since $\partial^2_{\xi_i \xi_j}H(\nabla u)$ is semipositively definite and by using Kato inequality 
$$
 a_{ij}\partial^2_{\xi_k \xi_l}H^2(\nabla u)u_{ik}u_{jl} \geq  a_{ij} \partial_{\xi_k}H(\nabla u) \partial_{\xi_l}H(\nabla u)u_{ik}u_{jl}
$$ 
(see \cite[Lemma 2.2]{WangXia}), we obtain that
\begin{multline} \label{manca2}
 pH^{p-1}(\nabla u)\partial_{\xi_{i}} H(\nabla u)\partial_{\xi_{j}} H(\nabla u)\partial^2_{\xi_l \xi_k}H(\nabla u)u_{ik}u_{jl}\\
 =pH^{p-1}(\nabla u)
  \partial^2_{\xi_l \xi_k}H(\nabla u)\left(\partial_{\xi_{i}} H(\nabla u)u_{ik}\right)
\left(\partial_{\xi_{j}} H(\nabla u)u_{jl}\right)\geq 0 \,.
\end{multline}
The term
\begin{equation*}
    H^{p}(\nabla u)\partial^2_{\xi_i \xi_j}H(\nabla u)\partial^2_{\xi_l \xi_k}H(\nabla u)u_{ik}u_{jl}
\end{equation*}
is nonnegative definite as well. Indeed, the matrix $\partial^2_{\xi_l \xi_k}H(\nabla u)$ has $N - 1$ strictly positive eigenvalues and one null
eigenvalue (\cite[Lemma 2.4]{CFV} and \cite[Lemma 2.5]{CFV}). Hence, we can write the
matrix
\begin{equation*}
  \partial^2_{\xi_l \xi_k}H(\nabla u)=O^T\Lambda O,
\end{equation*}
where the matrix $O$ is orthogonal and the matrix $\Lambda$ is diagonal and such that  
$$
\Lambda=\text{diag}(\mu_1, \mu_2,\ldots, \mu_{n-1},0) \,,
$$
with $\mu_i\geq 0$ for any $i = 1, \ldots, n-1$. Let $U =(u_{ij})_{i,j=1,\ldots,N}$ and $\widetilde{U}=OUO^T$. Then we have
\begin{eqnarray*}
H^{p}(\nabla u)\partial^2_{\xi_i \xi_j}H(\nabla u)\partial^2_{\xi_l \xi_k}H(\nabla u)u_{ik}u_{jl}&=&H^{p}(\nabla u)\text{Tr}\left(O^T\Lambda OUO^T\Lambda OU\right)\\
&=&
H^{p}(\nabla u)\text{Tr}\left(\Lambda OUO^T\Lambda OUO^T\right)\\
&=&H^{p}(\nabla u)\text{Tr}\left(\Lambda \widetilde{U}\Lambda \widetilde{U}\right) \,,
\end{eqnarray*}
and from the definition of $\Lambda$ and $\widetilde{U}$ we obtain
\begin{eqnarray*}
H^{p}(\nabla u)\partial^2_{\xi_i \xi_j}H(\nabla u)\partial^2_{\xi_l \xi_k}H(\nabla u)u_{ik}u_{jl}&=&
H^{p}(\nabla u) \mu_i \mu_j \widetilde{u}^2_{ij}\\ 
&=& H^{p}(\nabla u)\mu^2_i \widetilde{u}^2_{ii} + H^{p}(\nabla u)\sum_{i\neq k}\mu_i \mu_k \widetilde{u}^2_{ik}\geq H^{p}(\nabla u)\mu^2_i \widetilde{u}^2_{ii},
\end{eqnarray*}
where we used that $\mu_i \geq 0$. From Cauchy-Schwarz inequality, we obtain 
$$
H^{p}(\nabla u)\partial^2_{\xi_i \xi_j}H(\nabla u)\partial^2_{\xi_l \xi_k}H(\nabla u)u_{ik}u_{jl}
 \geq \dfrac{1}{N-1}H^{p}(\nabla u)\left( \mu_i \widetilde{u}_{ii}\right)^2  \,,
$$
and hence
\begin{equation} \label{manca3}
H^{p}(\nabla u)\partial^2_{\xi_i \xi_j}H(\nabla u)\partial^2_{\xi_l \xi_k}H(\nabla u)u_{ik}u_{jl} \geq \dfrac{B_1B_2}{N-1} \,.
\end{equation}
From  \eqref{right}, \eqref{left}, \eqref{manca1}, \eqref{manca2} and \eqref{manca3} we obtain \eqref{P_11}.
\end{proof}

%Let
%\begin{equation*}
%  C_{N,p}=\sqrt{\frac{2(N+p-2)}{N-1}}.
%\end{equation*}
\begin{lemma}\label{lemma_MaxPrinc_Du_u}
Let $E \subset \R^N$ be a domain and let $u$ be such that $\Delta_Hu=0$ in $E$. We have
\begin{equation}\label{P_1}
\mathcal{L} (H^2(\nabla u)) \geq \frac{2(N+p-2)(p-1) }{N-1}H^{p-2}(\nabla u)\left( \partial_{\xi_i} H(\nabla u) \partial_{\xi_j} H(\nabla u)  u_{ij}\right)^2 \,,
\end{equation}
where $\nabla u \neq0$.

Moreover if $E$ is bounded then $H(\nabla u)$ satisfies the maximum principle. 
\end{lemma}
\begin{proof}
We first notice the following Bochner formula
\begin{equation} \label{Bochner2}
a_{ij}\partial^2_{ij} H^2(\nabla u) = a_{ij}\partial^2_{\xi_k \xi_l}H^2(\nabla u)u_{ik}u_{jl}-a_{ijl}\partial_l H^2(\nabla u)u_{ij}
\end{equation}
where $\nabla u \neq 0$. Indeed, from \eqref{aij} and \eqref{finsleraijuj} and since $\Delta_p^Hu=0$ we have
\begin{eqnarray*}
a_{ij} \partial_{ij}^2 \left(H^2(\nabla u)\right)&=& a_{ij} \partial_{j}\left(\partial_{\xi_{k}} H^2 (\nabla u)u_{ik}\right)=a_{ij} \partial^2_{\xi_{k}\xi_{l}} H^2(\nabla u) u_{ik}u_{jl}+a_{ij} \partial_{\xi_{k}} H^2(\nabla u) u_{ijk}\\
&=&a_{ij} \partial^2_{\xi_{k}\xi_{l}} H^2(\nabla u) u_{ik}u_{jl}+ \partial_{\xi_{k}} H^2(\nabla u)\partial_{k}(\underbrace{a_{ij}u_{ij}}_{\Delta^H_p u=0})-\partial_{\xi_{k}} H^2(\nabla u)\partial_k(a_{ij})u_{ij}\\
&=&a_{ij} \partial^2_{\xi_{k}\xi_{l}} H^2(\nabla u) u_{ik}u_{jl}-2H(\nabla u)\partial_{\xi_{k}} H(\nabla u)a_{ijl}u_{lk}u_{ij} \,,
\end{eqnarray*}
where $\nabla u \neq 0$. Then
\begin{equation}\label{LH^2}
  \mathcal{L} (H^2(\nabla u))=a_{ij}\partial^2_{ij} H^2(\nabla u)+a_{ijl}\partial_l H^2(\nabla u)u_{ij}=a_{ij}\partial^2_{\xi_k \xi_l}H^2(\nabla u)u_{ik}u_{jl}\,,
\end{equation}
which proves \eqref{Bochner2}.

By differentiating
\begin{equation*}
  \Delta_p^Hu=0
\end{equation*}
with respect to $x_k$ where $\nabla u \neq 0$, we obtain
\begin{equation*}
  \partial_i\left((p-1)H^{p-2}(\nabla u)\partial_{\xi_{i}} H(\nabla u)\partial_{\xi_{j}}H(\nabla u)u_{jk}+ H^{p-1}(\nabla u)\partial^2_{\xi_{i}\xi_{j}} H(\nabla u)u_{jk}\right)=0.
\end{equation*}
We multiply the above equation by $H(\nabla u)\partial_{\xi_{k}} H(\nabla u)$ and obtain
\begin{eqnarray*}
  0&=&H(\nabla u)\partial_{\xi_{k}} H(\nabla u)\partial_i\left((p-1)H^{p-2}(\nabla u)\partial_{\xi_{i}} H(\nabla u)\partial_{\xi_{j}}H(\nabla u)u_{jk}+ H^{p-1}(\nabla u)\partial^2_{\xi_{i}\xi_{j}} H(\nabla u)u_{jk}\right)\\
  &=&\partial_i\left(H(\nabla u)\partial_{\xi_{k}} H(\nabla u)\left((p-1)H^{p-2}(\nabla u)\partial_{\xi_{i}} H(\nabla u)\partial_{\xi_{j}}H(\nabla u)u_{jk}+ H^{p-1}(\nabla u)\partial^2_{\xi_{i}\xi_{j}} H(\nabla u)u_{jk}\right)\right)\\
  &-&\partial_i\left(H(\nabla u)\partial_{\xi_{k}} H(\nabla u)\right)\left((p-1)H^{p-2}(\nabla u)\partial_{\xi_{i}} H(\nabla u)\partial_{\xi_{j}}H(\nabla u)u_{jk}+ H^{p-1}(\nabla u)\partial^2_{\xi_{i}\xi_{j}} H(\nabla u)u_{jk}\right),
\end{eqnarray*}
that is
%\begin{eqnarray*}
%  && \partial_i\left(H(\nabla u)\partial_{\xi_{k}}  H(\nabla u)\right)\Big((p-1)H^{p-2}(\nabla u)\partial_{\xi_{i}} H(\nabla u)\partial_{\xi_{j}}H(\nabla u)u_{jk}\nonumber\\
%  &&\quad + H^{p-1}(\nabla u)\partial^2_{\xi_{i}\xi_{j}} H(\nabla u)u_{jk}\Big)\nonumber\\
%&&  =\partial_i\Big(H(\nabla u)\partial_{\xi_{k}} H(\nabla u)\Big((p-1)H^{p-2}(\nabla u)\partial_{\xi_{i}} H(\nabla u)\partial_{\xi_{j}}H(\nabla u)u_{jk}\\
%  &&\quad + H^{p-1}(\nabla u)\partial^2_{\xi_{i}\xi_{j}} H(\nabla u)u_{jk}\Big)\Big)\nonumber\\
%\end{eqnarray*}
\begin{multline*}
\partial_i\left(H(\nabla u)\partial_{\xi_{k}}  H(\nabla u)\right)\Big((p-1)H^{p-2}(\nabla u)\partial_{\xi_{i}} H(\nabla u)\partial_{\xi_{j}}H(\nabla u)u_{jk} + H^{p-1}(\nabla u)\partial^2_{\xi_{i}\xi_{j}} H(\nabla u)u_{jk}\Big) \\
 =\partial_i\Big(H(\nabla u)\partial_{\xi_{k}} H(\nabla u)\Big((p-1)H^{p-2}(\nabla u)\partial_{\xi_{i}} H(\nabla u)\partial_{\xi_{j}}H(\nabla u)u_{jk} + H^{p-1}(\nabla u)\partial^2_{\xi_{i}\xi_{j}} H(\nabla u)u_{jk}\Big)\Big) \,,
\end{multline*}
and from the definition of $a_{ij}$, we have
%\begin{eqnarray}\label{secondline}
%  && \partial_i\left(H(\nabla u)\partial_{\xi_{k}}  H(\nabla u)\right)\Big((p-1)H^{p-2}(\nabla u)\partial_{\xi_{i}} H(\nabla u)\partial_{\xi_{j}}H(\nabla u)u_{jk} \\
%  &&\quad+ H^{p-1}(\nabla u)\partial^2_{\xi_{i}\xi_{j}} H(\nabla u)u_{jk}\Big)\nonumber\\
%&&=\partial_i\left(a_{ij}  H(\nabla u)\partial_{\xi_{k}} H(\nabla u) u_{jk}\right).\nonumber
%\end{eqnarray}
%
\begin{eqnarray}\label{secondline}
  && \partial_i\left(H(\nabla u)\partial_{\xi_{k}}  H(\nabla u)\right)\Big((p-1)H^{p-2}(\nabla u)\partial_{\xi_{i}} H(\nabla u)\partial_{\xi_{j}}H(\nabla u)u_{jk} \\
  &&\quad+ H^{p-1}(\nabla u)\partial^2_{\xi_{i}\xi_{j}} H(\nabla u)u_{jk}\Big) =\partial_i\left(a_{ij}  H(\nabla u)\partial_{\xi_{k}} H(\nabla u) u_{jk}\right).\nonumber
\end{eqnarray}
%
%\begin{multline}\label{secondline}
%\partial_i\left(a_{ij}  H(\nabla u)\partial_{\xi_{k}} H(\nabla u) u_{jk}\right) \\
%= \partial_i\left(H(\nabla u)\partial_{\xi_{k}}  H(\nabla u)\right)\Big((p-1)H^{p-2}(\nabla u)\partial_{\xi_{i}} H(\nabla u)\partial_{\xi_{j}}H(\nabla u)u_{jk} + H^{p-1}(\nabla u)\partial^2_{\xi_{i}\xi_{j}} H(\nabla u)u_{jk}\Big) \,.
%\end{multline}
%
Since
\begin{equation}\label{bimbo}
\partial_i\left( H(\nabla u)\partial_{\xi_{k}}H(\nabla u)\right)=\partial_{\xi_{l}} H(\nabla u)\partial_{\xi_{k}}H(\nabla u)u_{li}+ H(\nabla u)\partial^2_{\xi_{l}\xi_{k}} H(\nabla u)u_{li},
\end{equation}
we have
\begin{eqnarray*}
&&\partial_k\left( H(\nabla u)\partial_{\xi_{i}}H(\nabla u)\right)\partial_i\left( H(\nabla u)\partial_{\xi_{k}}H(\nabla u)\right)\\
&&=\Big[\partial_{\xi_{i}} H(\nabla u)\partial_{\xi_{j}}H(\nabla u)+H(\nabla u)\partial^2_{\xi_{i}\xi_{j}} H(\nabla u)\Big]\Big[\partial_{\xi_{l}} H(\nabla u)\partial_{\xi_{k}}H(\nabla u)u_{li}+ H(\nabla u)\partial^2_{\xi_{l}\xi_{k}} H(\nabla u)\Big]u_{li}u_{jk}\\
&&=\partial_{\xi_{i}} H(\nabla u)\partial_{\xi_{j}}H(\nabla u)\partial_{\xi_{l}}H(\nabla u)\partial_{\xi_{k}}H(\nabla u)u_{li}u_{jk}+
H(\nabla u)\partial_{\xi_{i}} H(\nabla u)\partial_{\xi_{j}}H(\nabla u)\partial^2_{\xi_{l}\xi_{k}}H(\nabla u)u_{li}u_{jk}\\
&&+H(\nabla u)\partial^2_{\xi_{i}\xi_{j}}H(\nabla u)\partial_{\xi_{l}} H(\nabla u)\partial_{\xi_{k}}H(\nabla u)u_{li}u_{jk}+
H^2(\nabla u)\partial^2_{\xi_{i}\xi_{j}}H(\nabla u)\partial^2_{\xi_{l}\xi_{k}}H(\nabla u)u_{li}u_{jk},
\end{eqnarray*}
From
$$\dfrac{1}{2}\mathcal{L}\left(H^2(\nabla u)\right) =\partial_i\left(a_{ij}H(\nabla u)\partial_{\xi_{k}} H(\nabla u)u_{jk}\right)\,,$$
and \eqref{secondline} we obtain
\begin{eqnarray*}
 \dfrac{1}{2}\mathcal{L}\left(H^2(\nabla u)\right)  &=& (p-1)H^{p-2}(\nabla u)\partial_{\xi_{i}} H(\nabla u)\partial_{\xi_{j}}H(\nabla u)\partial_i\left( H(\nabla u)\partial_{\xi_{k}}H(\nabla u)\right))u_{jk} \nonumber \\
 &&\quad+ H^{p-1}(\nabla u)\partial^2_{\xi_{i}\xi_{j}} H(\nabla u)\partial_i\left( H(\nabla u)\partial_{\xi_{k}}H(\nabla u)\right)u_{jk} \nonumber \\
  &=& \Big[(p-2)H^{p-2}(\nabla u)\partial_{\xi_{i}} H(\nabla u)\partial_{\xi_{j}}H(\nabla u)u_{jk}+
 H^{p-2}(\nabla u)\partial_{\xi_{i}} H(\nabla u)\partial_{\xi_{j}}H(\nabla u)u_{jk} \nonumber \\
 &&\quad+ H^{p-1}(\nabla u)\partial^2_{\xi_{i}\xi_{j}} H(\nabla u)u_{jk}\Big] \partial_i\left( H(\nabla u)\partial_{\xi_{k}}H(\nabla u)\right)\,,
 \end{eqnarray*}
and \eqref{bimbo} yields
\begin{eqnarray}\label{Lmezzi}
 \dfrac{1}{2}\mathcal{L}\left(H^2(\nabla u)\right)&=& (p-1)H^{p-2}(\nabla u)\left(\partial_{\xi_{i}} H(\nabla u)\partial_{\xi_{l}}H(\nabla u)u_{li}\right)^2 \nonumber \\
&&\quad+ (p-1)H^{p-1}(\nabla u)\partial^2_{\xi_{l}\xi_{k}} H(\nabla u)\left(\partial_{\xi_{i}} H(\nabla u)u_{li}\right)
\left(\partial_{\xi_{j}} H(\nabla u)u_{jk}\right) \nonumber \\
&&\quad+  H^{p-1}(\nabla u)\partial^2_{\xi_{i}\xi_{j}} H(\nabla u)\left(\partial_{\xi_{l}} H(\nabla u)u_{li}\right)
\left(\partial_{\xi_{k}} H(\nabla u)u_{jk}\right)\\
 &&\quad+  H^{p}(\nabla u)(\nabla u)\left(\partial^2_{\xi_{i}\xi_{j}} H(\nabla u)u_{jk}\right)\left(\partial^2_{\xi_{l}\xi_{k}} H(\nabla u)u_{li}\right).\nonumber
 \end{eqnarray}
From \eqref{P_11}, \eqref{LH^2} and \eqref{Lmezzi} we obtain \eqref{P_1}. 

We set $E_0=\{x \in E :\ \nabla u = 0\}$; since $u\in C^{1,\alpha}$ then $E_0$ is closed. From \eqref{P_11} we have that $H(\nabla u)^2$ satisfies a maximum principle in $E \setminus E_0$ and hence $\max H(\nabla u)^2$ is attained at $\partial E \cup \partial E_0$. Since $H(\nabla u)=0$ in $ E_0$, we have that $ H(\nabla u)$ attains the maximum at $\partial E$, which yields that the maximum principle holds.
\end{proof}

By using Lemmas \ref{lemma_maxprinc1} and \ref{lemma_MaxPrinc_Du_u} we can prove a maximum principle for a $P$-function which is suitable for estimating the blow-up of the gradient. In particular, it takes care of the presence of the neck $\mathcal{N}_{\delta}(w)$.

Let $f \in C^2 (\overline \Omega)$ be a cut-off function such that
\begin{equation} \label{f_1}
|f| =1 \text { in } \overline\Omega_\delta \setminus \mathcal{N}_{\delta}(w) \,, \quad f= 0 \text { in } \mathcal{N}_{\delta}\left(\frac w2 \right) \, .
\end{equation}
Moreover we choose $f$ such that
\begin{equation} \label{f_2}
 \frac{f}{w} \leq |\nabla f|^2 \quad \text{ and } \quad |\nabla^2 f | \leq \frac{1}{w^2} \,\text { in } \mathcal{N}_{\delta}(w) \setminus  \mathcal{N}_{\delta}\left(\frac w2 \right).
\end{equation}

\begin{theorem}\label{MaximumPrinciple}
Let $u_\delta$ be such that $\Delta^H_p u_\delta = 0$ in $\Omega_\delta$.
Let $f$ satisfy \eqref{f_1} and \eqref{f_2}.

There exists $\lambda_0=\lambda_0(\|f\|_{C^2} , \|H\|_{C^3(\partial B_H(0,1))})$, with $\lambda_0 = O(w^{-2})$ as $w \to 0^+$, such that the function
\begin{equation}\label{pfunction}
  P(x)= f(x) H(\nabla u_\delta)^2+ \lambda u^2_\delta
\end{equation}
satisfies the maximum principle for any $\lambda \geq \lambda_0$, i.e.
\begin{equation} \label{P_3}
\max_{x \in\overline{\Omega}_\delta} P(x) = \max_{x \in \partial \Omega_\delta} P(x)
\end{equation}
for $\lambda \geq \lambda_0$.
\end{theorem}
\begin{proof}
We first notice that if $P$ attains the maximum at a point $x_0$ such that $\nabla u(x_0)=0$ then $x_0 \in \partial \Omega$. Indeed, since $P(x_0) =  \lambda u_{\delta}(x_0)^2 $ we have
$$
f(x) H^2(\nabla u_\delta(x)) + \lambda u_\delta(x)^2 \leq \lambda u_{\delta}(x_0)^2
$$
for any $x \in \Omega_\delta$. In particular $|u_\delta(x)| \leq  |u_{\delta}(x_0)|$ for any $x \in \Omega_\delta$, and Lemma \ref{lemma_maxmin} yields that $x_0 \in \partial \Omega$.

Now let assume that $P$ attains the maximum at a point $x_0$ such that $\nabla u_\delta(x_0)\neq0$.
From \eqref{P_1} and \eqref{P_2} we have
\begin{eqnarray}\label{LP}
  \mathcal{L}P &=& a_{ij}\partial_{ij} f(x) H^2(\nabla u_\delta) +2a_{ij}\partial_i f(x)\partial_{j}H^2(\nabla u_\delta)+ a_{ijl} \partial_l f(x) H^2(\nabla u_\delta) u_{ij}\nonumber \\
  & & \quad + f(x)\left[a_{ij}\partial^2_{ij}H^2(\nabla u_\delta)+ a_{ijl} \partial_l H^2(\nabla u_\delta) u_{ij}\right]+2\lambda(p-1)H^p(\nabla u_\delta)\nonumber\\
  &\geq& a_{ij}\partial_{ij} f(x) H^2(\nabla u_\delta) +2a_{ij}\partial_i f(x)\partial_{j}H^2(\nabla u_\delta)+ a_{ijl} \partial_l f(x) H^2(\nabla u_\delta) u_{ij} \\
  & & \quad +f(x)(p-1)C^2_{N,p}H^{p-2}(\nabla u_\delta)\left( \partial_{\xi_i} H(\nabla u_\delta) \partial_{\xi_j} H(\nabla u_\delta)  u_{ij}\right)^2\nonumber \\
  & & \quad +2\lambda(p-1)H^p(\nabla u_\delta) \,,\nonumber 
  \end{eqnarray}
 where we set
$$
C_{N,p} =\sqrt{\frac{2(N+p-2)(p-1) }{N-1}} \,.
$$ 
Since $H$ is 1-homogeneous, the quantities $a_{ij}H^{2-p}(\nabla u_\delta)$, $a_{ijl}H^{3-p}(\nabla u_\delta)$ and $\partial_{\xi_i} H$ are 0-homogeneous. Hence there exists  $C_0$ depending only on $\|H\|_{C^3(\partial B_H(0,1))}$ such that
\begin{equation}\label{limitati}
|a_{ij} H^{2-p}(\nabla u_\delta)|, |a_{ijl}H^{3-p}(\nabla u_\delta)| \leq C_0 \,,\text{and} \, C_0^{-1}\leq |\partial_{\xi_i} H(\nabla u_\delta)| \leq C_0.
\end{equation}
From \eqref{LP}, \eqref{limitati} and by using Cauchy-Schwarz inequality, we have
\begin{eqnarray*}
  \mathcal{L}P &\geq&\left[\lambda(p-1) - C_0 \|\nabla^2 f\|_{C^0}\right]H^p(\nabla u_\delta)\\
   && \quad + \left[2(p-1)C_{N,p} \sqrt{\lambda f}C^{-2}_0- C\|\nabla^2 f\|_{C^0} \right] H^{p-1}(\nabla u_\delta)\left|\nabla^2 u_\delta\right|.
\end{eqnarray*}
By choosing $\lambda_0$ large enough we obtain that $\mathcal{L}P \geq 0$ for any $\lambda \geq \lambda_0$, and a simple calculation shows that  $\lambda_0$ depends only on $\|H\|_{C^3(\partial B_H(0,1))}$ and $\|f\|_{C^2}$, and that $\lambda_0 = O(w^{-2})$.  
\end{proof}

\begin{remark}\label{remark_Pfunction}
We mention that there are other maximum principles for $H(\nabla u)$ available in literature. In particular, one can prove that 
\begin{equation} \label{gonnarosa}
\mathcal{L} H(\nabla u)^p \geq 0
\end{equation} 
(see for instance \cite{HS2} and \cite{Barbu_Enache}). Since $\mathcal{L}$ is associated to the $p$-Laplace equation, \eqref{gonnarosa} may appear to be more natural to be considered. Unfortunately, \eqref{gonnarosa} does not serve to our purposes, in particular it can not be employed in the proof of Theorem \ref{MaximumPrinciple} due to the presence of the term $u^2$ in the $P$-function (another power of $u$ would produce an analogous problem). This is the reason why we considered the quantity $\mathcal{L} H(\nabla u)^2$ in Lemma \ref{lemma_MaxPrinc_Du_u}.
\end{remark}

\section{Bounds for the gradient outside the neck} \label{sect_unif_bounds}
As we have mentioned in the Introduction, we will show the following behaviour of the gradient of $u_\delta$: (i) it may have a blow-up at the point where the two inclusions touch and (ii) it remains uniformly bounded far from that point. In this section we prove (ii), while (i) will be proved in Section \ref{sect_thm1}. 

We first notice that the gradient of $u_\delta$ is uniformly bounded on $\partial \Omega$ independently of $\delta$, i.e.
%\begin{lemma} \label{lemma_bound_partial_Omega}
that there exists a constant $C > 0$ independent of $\delta$ such that
\begin{equation}\label{UniformeLimitatezza1}
  \max_{\partial \Omega}H(\nabla u_{\delta})\leq C.
\end{equation}

Indeed, \eqref{UniformeLimitatezza1} follows from the following argument. We can choose a smooth domain $A \subset \Omega$ such that the inclusions $D_\delta^1$ and $D_\delta^2$ are contained in $A$ and such that $A$ is far from $\partial \Omega$ more than $K/2$. The uniform bound on the gradient on $\partial \Omega$ can be obtained by comparison principle, in particular by comparing $u_\delta$ to $v_*$ and $v^*$, where $v_*$ and $v^*$ are the solutions to
$$
\begin{cases}
\Delta^H_p v_* = 0 &  \text{ in } \Omega \setminus \overline{A} \,, \\
v_*= \varphi &  \text{ on } \partial \Omega \,, \\
v_*= \min_{\partial \Omega} \varphi & \text{ on } \partial A \,,
\end{cases}
$$
and
$$
\begin{cases}
\Delta^H_p v^* = 0 &  \text{ in } \Omega \setminus \overline{A} \,, \\
v^*= \varphi &  \text{ on } \partial \Omega \,, \\
v^*= \max_{\partial \Omega} \varphi & \text{ on } \partial A \,,
\end{cases}
$$
respectively. 

It is clear that $v_*$ and $v^*$ are a lower and an upper barrier for $u_\delta$, respectively, at any point on $\partial \Omega$. Hence, the normal derivative of $u_\delta$ can be bounded in terms of the gradient of $v_*$, $v^*$, and thus $H(u_\delta)$ can be bounded by some constant $C$ which depends only on $K$ and $\varphi$, which implies \eqref{UniformeLimitatezza1}.

\medskip

Now we show that the gradient is uniformly bounded on the boundary of the inclusions at the points which are not in the neck.

\begin{lemma} \label{lemma_bound_partial_B}
Let $u_{\delta}$ be the solution of \eqref{DpH} and let $w>0$ be fixed. There exists a constant $C > 0$ independent of $\delta$ such that
\begin{equation}\label{UniformeLimitatezza}
  \max_{\partial D^i_{\delta} \setminus \partial \mathcal{N}_{\delta}(w)} H(\nabla u_{\delta})\leq C \,, \quad \quad i=1,2.
\end{equation}
\end{lemma}

\begin{proof}
Let $z \in \partial D^1_{\delta} \setminus \partial \mathcal{N}_{\delta}(w)$ be fixed.  Let  $B_{H_0}(z_0,r_1)$ be the interior touching ball to $\partial D^1_{\delta}$ at $z$ (hence $r_1 < R_1$), and define $B_{H_0}(z_0,r_2)$ as the exterior touching ball to $D^2_{\delta}$ centered at $z_0$. 

The proof consists in comparing the solution $u_\delta$ to an upper barrier $\overline{v}$ and a lower barrier $\underline{v}$ for $u_\delta$ at $z$, which are defined as the solutions to
\begin{equation*}
\begin{cases}
\Delta^H_p \overline{v} = 0 & \text{in } B_{H_0}(z_0,r_2) \setminus \overline{B}_{H_0}(z_0,r_1)\,, \\
\overline{v}=\mathcal{U}^1_{\delta} & \text{on } \partial B_{H_0}(z_0,r_1),\\
\overline{v}=\displaystyle\max_{\partial \Omega} \varphi & \text{on } \partial B_{H_0}(z_0,r_2),
\end{cases}
\end{equation*}
and
\begin{equation*}
\begin{cases}
\Delta^H_p \underline{v} = 0 & \text{in } B_{H_0}(z_0,r_2) \setminus \overline{B}_{H_0}(z_0,r_1)\,, \\
\underline{v}=\mathcal{U}^1_{\delta} & \text{on } \partial B_{H_0}(z_0,r_1),\\
\underline{v}=\displaystyle\min_{\partial \Omega} \varphi & \text{on } \partial B_{H_0}(z_0,r_2),
\end{cases}
\end{equation*}
respectively, where $\mathcal{U}^i_{\delta}$ are defined in \eqref{potenziale} (see \eqref{v_anello} for the explicit expression).
By evaluating the gradient of $\overline v$ and $\underline v$ at $z$ and by using \eqref{LemmaCianchiSalani11}, we have
\begin{equation} \label{frida1}
H(\nabla\overline{v}(z))=
\begin{cases}
\dfrac{N-p}{p-1}\left|\mathcal{U}^1_{\delta}-\displaystyle\max_{\partial \Omega} \varphi\right|\dfrac{r_1^{\frac{1-N}{p-1}}}{r_1^{\frac{p-N}{p-1}}-r_2^{\frac{p-N}{p-1}}} & \text{if } 1<p<N, \\
\\
\left|\mathcal{U}^1_{\delta}-\displaystyle\max_{\partial \Omega} \varphi\right|\dfrac{r_1^{-1}}{\ln \left(r_2r_1^{-1}\right)}& \text{if } N=p,
\end{cases}
\end{equation}
and
\begin{equation}\label{frida2}
H(\nabla\underline{v}(z))=
\begin{cases}
\dfrac{N-p}{p-1}\left|\mathcal{U}^1_{\delta}-\displaystyle\min_{\partial \Omega} \varphi\right|\dfrac{r_1^{\frac{1-N}{p-1}}}{r_1^{\frac{p-N}{p-1}}-r_2^{\frac{p-N}{p-1}}} & \text{if } 1<p<N, \\
\\
\left|\mathcal{U}^1_{\delta}-\displaystyle\min_{\partial \Omega} \varphi\right|\dfrac{r_1^{-1}}{\ln \left(r_2 r_1^{-1} \right)}& \text{if } N=p.
\end{cases}
\end{equation}
Let $r_1=c w$ for some small constant $c>0$. Since $z$ is not in the neck, there exists a constant $\alpha>1$ such that $r_2 \geq \alpha r_1$ for any $\delta \geq 0$, with $\alpha$ not depending on $\delta$ (the constant $\alpha$ can be explicitely calculated by considering the limit configuration for $\delta=0$). Hence we have that
\begin{equation} \label{frida3}
\dfrac{r_1^{\frac{1-N}{p-1}}}{r_1^{\frac{p-N}{p-1}}-r_2^{\frac{p-N}{p-1}}} \leq \frac{1}{cw\left(1 - \alpha^{\frac{p-N}{p-1}}\right)} \quad \text{ for } 1<p<N \,,
\end{equation}
and
\begin{equation} \label{frida4}
\dfrac{r_1^{-1}}{\ln \left(r_2r_1^{-1} \right)} \leq \frac{1}{cw \ln \alpha} \quad \text{ for } N=p \,.
\end{equation}
From \eqref{frida1}-\eqref{frida4} and Lemma \ref{lemma_maxmin} we find that 
$$
H(\nabla u_\delta (z)) \leq C
$$
where $C$ depends only on the dimension $N,p$, $\|\varphi\|_{C^0(\partial \Omega)}$ and $w$, and does not depends on $\delta$.
\end{proof}

The arguments used for proving \eqref{UniformeLimitatezza1} and Lemma \ref{lemma_bound_partial_B} can be used to prove that, in the limit case $\delta=0$, the gradient of $u_0$ remains uniformly bounded. Indeed, from  Lemma \ref{lemma_MaxPrinc_Du_u} we have that $H(\nabla u_0)$ attains the maximum on $\partial \Omega_0 = \partial \Omega \cup \partial D^1_0 \cup \partial D^2_0$. From the third condition in \eqref{PH0}, the maximum and minimum of Lemma $u_0$ are attained on $\partial \Omega$. Hence, the bound on $\partial \Omega$ can be obtained as done for  \eqref{UniformeLimitatezza1}.  The bound on $\partial D^1_0$ (and analogously the one on $\partial D^2_0$) can be obtained by comparison principle, more precisely by comparing $u_0$ and $v_1$ and $v_2$, where $v_i$ is the solution to $\Delta^H_p v_i=0$ in $\Omega \setminus \overline D^1_0$, $v_i=u_0$ on $\partial D^i_0$, $i=1,2$ $v_1=\max_{\partial \Omega} \phi$ and $v_2=\min_{\partial \Omega} \phi$ on $\partial \Omega$. Thus
\begin{equation} \label{u0_W1infty}
H(\nabla u_0) \leq C
\end{equation}
in $\overline \Omega_0$.

\medskip

Now we show that the gradient is bounded outside the neck, and we make explicit the dependency on $w$.

\begin{lemma} \label{lemma_bound_Omega_neck}
Let $u_{\delta}$ be the solution of \eqref{DpH} and let $w>0$. There exists a constant $C > 0$ independent of $\delta$ and $w$ such that
\begin{equation}\label{UniformeLimitatezza}
  \max_{\overline\Omega_\delta \setminus \mathcal{N}_{\delta}(w)}H(\nabla u_{\delta})\leq \frac C w.
\end{equation}
\end{lemma}

\begin{proof}
The maximum principle for the gradient of $u_\delta$ (see Lemma \ref{lemma_MaxPrinc_Du_u}) yields that the maximum of $H(\nabla u)$ in $\overline\Omega_\delta \setminus \mathcal{N}_{\delta}(w)$ is attained on $\partial (\Omega_\delta \setminus \mathcal{N}_{\delta}(w))$.

From \eqref{UniformeLimitatezza1} and Lemma \ref{lemma_bound_partial_B}, we only need to prove uniform bounds for $H(\nabla u)$ on $\partial \mathcal{N}^{\pm}_{\delta}(w)$, where
\begin{equation}\label{Npm}
  \partial \mathcal{N}^{\pm}_{\delta}(w)= \partial \mathcal{N}_{\delta}(w)\cap \{|\QQ x'|=\pm w\} \,.
\end{equation}

Let $P$ be as in Theorem \ref{MaximumPrinciple} (see formula \eqref{pfunction}). From \eqref{f_1} we have that
$$
\max_{\partial\mathcal{N}^{\pm}_{\delta}(w)} H^2(\nabla u_{\delta})
 = \max_{\partial\mathcal{N}^{\pm}_{\delta}(w)}f(x) H^2(\nabla u_{\delta})\leq \max_{\partial\mathcal{N}^{\pm}_{\delta}(w)} P(x) \leq \max_{\overline \Omega_\delta} P(x) \,,
 $$
and Theorem \ref{MaximumPrinciple} implies that there exists a constant $\lambda_0 = O(w^{-2})$ such that \eqref{pfunction} satisfies the maximum principle for any $\lambda \geq \lambda_0$ and we obtain 
$$
\max_{\partial\mathcal{N}^{\pm}_{\delta}(w)} H^2(\nabla u_{\delta}) \leq \max_{\overline \Omega_\delta} P(x) = \max_{\partial \Omega_\delta} P(x) \,.
$$
Since $\|u_\delta\|_{C^0(\Omega_\delta)} \leq \|\varphi\|_{C^0(\partial \Omega)}$ (see Lemma \ref{lemma_maxmin}) and $\lambda_0=O(w^{-2})$ (see Theorem \ref{MaximumPrinciple}), we have that there exists a constant $C$ independent of $\delta$ and $w$ such that
\begin{equation*}
  P(x)= f(x) H^2(\nabla u_{\delta})+ \lambda u_{\delta}^2\leq f(x) H^2(\nabla u_{\delta}) + Cw^{-2},
\end{equation*}
and hence
\begin{equation*}
\max_{\partial\mathcal{N}^{\pm}_{\delta}(w)} H^2(\nabla u_{\delta})\leq \max_{\partial\Omega_{\delta}}P(x)\leq \max_{\partial\Omega_{\delta}} \left[f(x) H^2(\nabla u_{\delta}) \right]+ Cw^{-2}.
\end{equation*}
Since $f=0$ in $\mathcal{N}_{\delta}(w/2)$, from  \eqref{UniformeLimitatezza1} and Lemma \ref{lemma_bound_partial_B} we find  that there exists a constant $C$ independent on $\delta$ and $w$ such that
\begin{equation} \label{grad_bound_Pi_pm}
  \max_{\partial \mathcal{N}^{\pm}_{\delta}(w)} H\left(\nabla u_{\delta} \right)\leq \frac{C}{w},
\end{equation}
which completes the proof.
\end{proof}
%
%Before giving the relation between $u_{\delta}$ and $u_{0}$ (see Proposition \ref{prop_udelta_conv_u0} below), in the next Lemma we show that gradient of $u_{0}$ is bounded.

%\begin{lemma} \label{lem_u0_W1infty}
%Let $u_0$ be the solution to \eqref{PH0}. Then $H(\nabla u_0) \leq C$.
%\end{lemma}
%
%\begin{proof}
%The proof is analogous to the ones of Lemmas  \ref{lemma_bound_partial_Omega} and \ref{lemma_bound_partial_B}, and we only give a sketch. Since $H(\nabla u_0)$ attains the maximum at the boundary (see Lemma \ref{lemma_MaxPrinc_Du_u}), we have to prove that $H(\nabla u_0)$ is bounded on $\partial \Omega$ and on $\partial D^1_0 \cup \partial D^2_0$. First we recall that, in view of the third condition in \eqref{PH0}, the maximum and minimum of Lemma $u_0$ are attained at $\partial \Omega$. Hence, the bound on $\partial \Omega$ can be obtained as in the proof of Lemma \ref{lemma_bound_partial_Omega}. The bound on $\partial D^1_0$ (and analogously the one on $\partial D^2_0$) can be obtained by comparison principle, more precisely by comparing $u_0$ and $v_1$ and $v_2$, where $v_i$ is the solution to $\Delta^H_p v_i=0$ in $\Omega \setminus \overline D^1_0$, $v_i=u_0$ on $\partial D^i_0$, $i=1,2$ $v_1=\max_{\partial \Omega} \phi$ and $v_2=\min_{\partial \Omega} \phi$ on $\partial \Omega$.
%\end{proof}

We conclude this section by giving the relation between $u_\delta$ and $u_0$. We first notice that, from Lemma \ref{lemma_bound_Omega_neck} and \cite[Theorem 2]{Dibenedetto}, for any fixed $w>0$ we have that there exists $\alpha>0$ independent of $\delta$ such that
\begin{equation*}  
\|u_\delta\|_{C^{1,\alpha}(\mathcal{K})} \leq C \qquad \text{ for any compact set } \mathcal{K} \subset \Omega_\delta \setminus \overline{\mathcal{N}}_{\delta}(w) \,,
\end{equation*}
where $C$ is a constant independent of $\delta$. This implies that $u_\delta$ converges to $u_0$ in $C^{1,\alpha}$ outside the neck, as shown in the following proposition. 

\begin{proposition} \label{prop_udelta_conv_u0}
Let $u_{\delta}$ be the solution of \eqref{DpH} and $u_0$ be the solution of \eqref{PH0}.

There exists a constant $0<\alpha <1 $ not depending on $\delta$ such that
\begin{equation} \label{u_delta_to_u_0}
  \lim_{\delta\rightarrow 0}\|u_{\delta} -u_{0}\|_{C^{1,\alpha}(E)}=0,
\end{equation}
for any compact set $E \subset \Omega_0$. Moreover, for any $i = 1, 2$ and for any neck $\mathcal{N}_\delta(w)$ of (sufficiently small) width $w$ we have
\begin{equation} \label{ferragosto}
   \lim_{\delta\rightarrow 0} \displaystyle\int_{\partial D^i_{\delta}\setminus \partial \mathcal{N}_{\delta}(w)} H^{p-1}\left(\nabla u_{\delta}\right)\nabla_{\xi}H\left(\nabla u_{\delta}\right)\cdot \nu ds= \displaystyle\int_{\partial D^i_0\setminus \partial \mathcal{N}_{\delta}(w)} H^{p-1}\left(\nabla u_{0}\right)\nabla_{\xi}H\left(\nabla u_{0}\right)\cdot \nu ds \,.
\end{equation}
\end{proposition}

\begin{proof}
The proof is analogous to the one of \cite[Proposition 4.5]{CiraoloSciammetta} (see also  \cite[Proposition 2.1]{GorbNovikov}), and we prefer to omit the details.
\end{proof}

\section{Proof of Theorem \ref{thm_main_1}}\label{sect_thm1}
\begin{lemma}\label{lemmalimite}
Let $\delta,w>0$ and let $\mathcal{R}_0$ and $\mathcal{N}_\delta(w)$ be given by \eqref{R0_def} and \eqref{neck}, respectively.  Let $u_\delta$ be the solution to \eqref{DpH} and define
\begin{equation}\label{I_1_def}
I_\delta(w)= \int_{\partial D_\delta^1 \cap \partial  \mathcal{N}_{\delta}(w)} H^{p-1}\left(\nabla u_{\delta}\right) \nabla_{\xi} H\left(\nabla u_{\delta}\right) \cdot \nu ds  \,.
\end{equation}
There exists $C>0$ independent of $\delta$ and $w$ such that
\begin{equation}\label{I1R0}
\lim_{\delta \to 0} |I_\delta(w) - \mathcal{R}_0| \leq C w^{N-1} \,.
\end{equation}
\end{lemma}

\begin{proof}
Since $u_\delta$ is the solution to \eqref{DpH}, the divergence theorem yields
\begin{equation} \label{843}
I_\delta (w)= \int_{\partial D_\delta^1} H^{p-1}\left(\nabla u_{\delta}\right) \nabla_{\xi} H\left(\nabla u_{\delta}\right) \cdot \nu ds  -  \int_{\partial D_\delta^1 \setminus \partial  \mathcal{N}_{\delta}(w)} H^{p-1}\left(\nabla u_{\delta}\right) \nabla_{\xi} H\left(\nabla u_{\delta}\right) \cdot \nu ds \,.
\end{equation}
We introduce a smooth auxiliary set $E$ containing $D^1_\delta$ and not containing $D^2_\delta$ such that $\partial E$ coincides with $\partial D_0^1$ in a neck of fixed width $w_0>0$. More precisely, $E$ is such that $D_\delta^1 \subset E$, $D_\delta^2 \subset \Omega \setminus E$ for any $\delta \geq 0$,  and  $\partial E \cap \mathcal{N}_{\delta}(w)  \subset \partial D_0^1$ for $w \leq w_0$. For simplicity of notation, we set
$$
\partial E^1(w) := \partial E \cap \mathcal{N}_{\delta}(w) \quad  \text{ and } \quad \partial E^2(w) := \partial E \setminus \mathcal{N}_{\delta}(w) \,.
$$
Since $H(\nabla u_0)$ is uniformly bounded in $\Omega_0$ (see \ref{u0_W1infty}), then
\begin{equation} \label{zanzara}
\Big{|} \int_{\partial E^1(w)} H^{p-1}(\nabla u_0)  \nabla_{\xi} H\left(\nabla u_{0}\right) \cdot \nu ds - \mathcal{R}_0 \Big{|} \leq C w^{N-1} \,,
\end{equation}
for some constant $C$ independent of $\delta$ and $w$.

By using that $\Delta_p^H u_\delta = 0$ and applying the divergence theorem in $E \setminus (  \mathcal{N}_{\delta}(w) \cup D_\delta^1)$ we obtain
\begin{multline*}
\int_{\partial E^1(w)} H^{p-1}\left(\nabla u_{\delta}\right) \nabla_{\xi} H\left(\nabla u_{\delta}\right) \cdot \nu ds \\
= \int_{E \cap \left(\partial\mathcal{N}^{+}_{\delta}(w))\cup \partial\mathcal{N}^{-}_{\delta}(w))\right)} H^{p-1}\left(\nabla u_{\delta}\right) \nabla_{\xi} H\left(\nabla u_{\delta}\right) \cdot \nu ds \\
+ \int_{\partial D_\delta^1 \setminus \partial  \mathcal{N}_{\delta}(w)} H^{p-1}\left(\nabla u_{\delta}\right) \nabla_{\xi} H\left(\nabla u_{\delta}\right) \cdot \nu ds \,,
\end{multline*}
where $\mathcal{N}^{\pm}_{\delta}(w)$ are defined by \eqref{Npm}.  Lemma \ref{lemma_bound_Omega_neck} yields
$$
\Big{|} \int_{E \cap \left(\partial\mathcal{N}^{+}_{\delta}(w))\cup \partial\mathcal{N}^{-}_{\delta}(w))\right)} H^{p-1}\left(\nabla u_{\delta}\right) \nabla_{\xi} H\left(\nabla u_{\delta}\right) \cdot \nu ds \Big{|} \leq \frac{C}{w^{p-1}} \delta \,;
$$
the third condition in \eqref{DpH} implies
$$
I_\delta = \int_{\partial D_\delta^1 \setminus \partial  \mathcal{N}_{\delta}(w)} H^{p-1}\left(\nabla u_{\delta}\right) \nabla_{\xi} H\left(\nabla u_{\delta}\right) \cdot \nu ds \,.
$$
Hence
\begin{equation} \label{zanza_morta}
\Big{|} I_\delta - \int_{\partial E^1(w)} H^{p-1}\left(\nabla u_{\delta}\right) \nabla_{\xi} H\left(\nabla u_{\delta}\right) \cdot \nu ds \Big{|} \leq \frac{C}{w^{p-1}} \delta \,.
\end{equation}
From \eqref{zanzara}, \eqref{zanza_morta} and Proposition \ref{prop_udelta_conv_u0} we obtain \eqref{I1R0}.
\end{proof}

Let $P \in \partial D_\delta^1 \cap \partial \mathcal{N}_\delta(w)$. Let $B_{H_0}(y_0,r)$ be a ball of center $y_0 \in B_\delta^1$ and radius $r$. We shall make use of the following values of the radius $r$. Let $r_1$ be such that $B_{H_0}(y_0,r)$ is tangent to $\partial D_\delta^1$ at $P$ from the inside, and denote by $r_2$ the radius of the concentric touching ball tangent to $\partial D_\delta^2$ from the outside. Analogously, let $B_{H_0}(z_0,\rho)$ be a ball of center $z_0 \in B_\delta^2$ and radius $\rho$. The radius $\rho$ will be chosen to be $\rho_1$ or $\rho_2$, which are defined as follows.
Let  $\rho_2$ be such that $B_{H_0}(z_0,\rho)$  touches  $\partial D_\delta^1$ at $P$ form the outside, and let $\rho_1$ be the radius of the concentric ball touching $\partial D_\delta^2$ from the inside (see Figure \ref{r1r1pho1pho2}). In the following lemma we establish pointwise bounds on the quantity inside $I_\delta$ in terms of $r_1,r_2,\rho_1$ and $\rho_2$. We assume that $\mathcal{U}_\delta^1 - \mathcal{U}_\delta^2 \geq 0$; the case $\mathcal{U}_\delta^1 - \mathcal{U}_\delta^2\leq 0$ is completely analogous.
\begin{center}
	\begin{figure}[htpb]
		\includegraphics[scale=0.5]{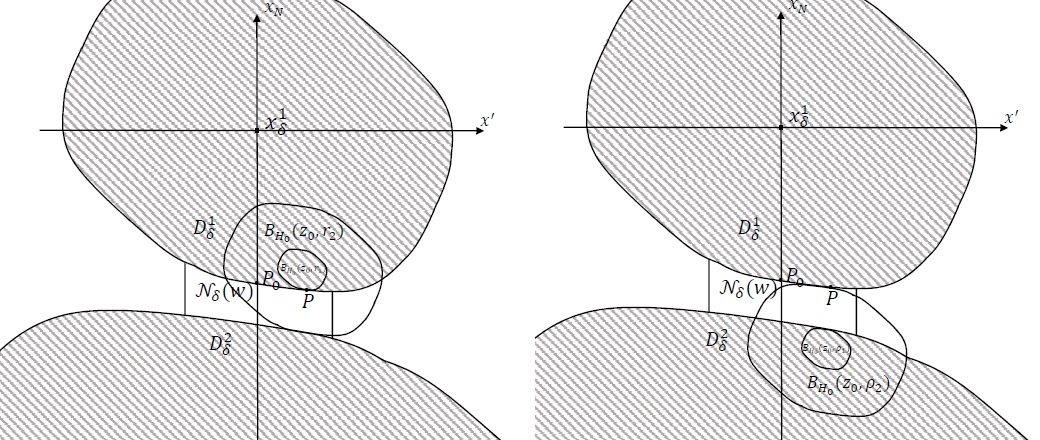}\label{r1r1pho1pho2}
		\caption{The choice of $r_1,r_2,\rho_1,\rho_2$.}
	\end{figure}
\end{center}
\begin{lemma} \label{lemma52}
Let assume that $\mathcal{U}^1_{\delta} \geq \mathcal{U}^2_{\delta}$. There exists $C>0$ independent of $\delta$ such that
\begin{equation} \label{lower_upper_bound}
 - \left(\dfrac{\mathcal{U}^1_{\delta}-\mathcal{U}^2_{\delta}}{\rho_2-\rho_1} \right)^{p-1}-  C  \leq H^{p-1}(\nabla u_\delta(P)) \nabla_\xi H(\nabla u_\delta (P)) \cdot \nu (P) \leq  - \left(\dfrac{\mathcal{U}^1_{\delta}-\mathcal{U}^2_{\delta}}{r_2-r_1} \right)^{p-1}+  C
\end{equation}
holds for any $P \in \partial D_\delta^1 \cap \partial \mathcal{N}_\delta(w)$.
\end{lemma}

\begin{proof}
Let $\underline v$ be given by
\begin{equation*}
\underline{v}(x)=
\begin{cases}
(\mathcal{U}^1_{\delta}-\mathcal{U}^2_\delta)\dfrac{H_0(x-y_0)^{\frac{p-N}{p-1}}-r_2^{\frac{p-N}{p-1}}}{r_1^{\frac{p-N}{p-1}}-r_2^{\frac{p-N}{p-1}}}+\mathcal{U}^2_\delta & \text{if } 1<p<N\,, \\
 & \\
\displaystyle(\mathcal{U}^1_{\delta}-\mathcal{U}^2_\delta)\dfrac{\ln (r_2^{-1} H_0(x-y_0))}{\ln (r_2^{-1}r_1 )}+\mathcal{U}^2_\delta  & \text{if }  N=p\,.
\end{cases}
\end{equation*}
Notice that $\Delta^H_p \underline v=0$ in $\mathbb{R}^N \setminus \{y_0\}$ and $\underline v = \mathcal{U}_\delta^i$ on $\partial B_{H_0}(y_0,r_i)$, $i=1,2$. Let
\begin{equation}
\mathcal{M} = \begin{cases}\label{M}
\dfrac{\mathcal{U}^1_{\delta}-\mathcal{U}^2_\delta}{r_1^{\frac{p-N}{p-1}}-r_2^{\frac{p-N}{p-1}}} & \text{if } 1<p<N\,, \\
 & \\
\dfrac{\mathcal{U}^1_{\delta}-\mathcal{U}^2_\delta}{\ln (r_2 r_1^{-1} )}  & \text{if }  N=p\,,
\end{cases}
\end{equation}
and observe that $\mathcal{M}>0$ for $1<p\leq N$ and $\mathcal{M}<0$ for $p>N$.

We notice that we can find a constant $M$, not depending on $\delta$, such that if $|\mathcal{M}|>M$, then $\underline v$ is a lower barrier for $u_\delta$. In this case, since
\begin{equation*}
H(\nabla\underline{v}(P))=
\begin{cases}
 \dfrac{\mathcal{U}^1_{\delta}-\mathcal{U}^2_{\delta}}{ r_1^{\frac{p-N}{p-1}}-r_2^{\frac{p-N}{p-1}}  }\dfrac{N-p}{p-1}  r_1^{\frac{1-N}{p-1}}& \text{if } 1<p<N\,, \\
 & \\
\dfrac{\mathcal{U}^1_{\delta}-\mathcal{U}^2_{\delta}}{\ln(r_2r_1^{-1})}\dfrac{1}{r_1}  & \text{if }  N=p\,,
\end{cases}
\end{equation*}
from the mean value theorem we have that there exists $\bar r \in (r_1,r_2)$ such that
\begin{equation*}
H(\nabla\underline{v}(P)) =
\dfrac{\mathcal{U}^1_{\delta}-\mathcal{U}^2_{\delta}}{r_2-r_1} \left( \frac{r_1}{\bar r} \right)^{\frac{1-N}{p-1}}
\end{equation*}
for any $p>1$, and hence
\begin{equation} \label{nabla_v_lower_bound}
H(\nabla\underline{v}(P)) \geq
\dfrac{\mathcal{U}^1_{\delta}-\mathcal{U}^2_{\delta}}{r_2-r_1} \,.
\end{equation}

Thanks to \eqref{nabla_v_lower_bound} we can give an upper bound on the quantity $H^{p-1}(\nabla u_\delta(P)) \nabla_\xi H(\nabla u_\delta (P)) \cdot \nu(P)$. Indeed, since $\underline v$ is a lower barrier for $u_\delta$ and $\partial D_\delta^1$ is a level line for $u_\delta$ then
\begin{equation}\label{-1}
\nabla_\xi H(\nabla u_\delta (P)) \cdot \nu(P) = - \nabla_\xi H(\nu(P) ) \cdot \nu(P) = - H(\nu(P) ) = -1 \,,
\end{equation}
where the last equality holds because $P$ lies on a Wulff shape. From \eqref{nabla_v_lower_bound} and \eqref{-1} we find
\begin{equation} \label{artista}
H^{p-1}(\nabla u_\delta(P)) \nabla_\xi H(\nabla u_\delta (P)) \cdot \nu (P) \leq - \left(\dfrac{\mathcal{U}^1_{\delta}-\mathcal{U}^2_{\delta}}{r_2-r_1}\right)^{p-1} \,.
\end{equation}
If $\mathcal{M} \leq M$, from elliptic estimates we have $H(\nabla u_\delta) \leq C$, where $C$ does not depends on $\delta$. Indeed, from the mean value theorem we have
$$
 \dfrac{\mathcal{U}^1_{\delta}-\mathcal{U}^2_{\delta}}{r_2-r_1} \leq M \dfrac{N-p}{p-1}r_1^{\frac{1-N}{p-1}} \,.
$$
Since $\partial D_\delta^1$ is of class $C^3$, $u_\delta$ is constant on $\partial D_\delta^1$, $|\mathcal{U}^1_{\delta}-\mathcal{U}^2_{\delta} | \leq C (r_2-r_1)$, and the distance of $P$ from $\partial D_\delta^2$ is of size $r_2-r_1$, from interior regularity estimates \cite{Dibenedetto} we have that  $H(\nabla u_\delta) \leq C$, where $C$ does not depends on $\delta$.

Hence
\begin{equation*}
H^{p-1}(\nabla u_\delta(P)) \nabla_\xi H(\nabla u_\delta (P)) \cdot \nu (P) \leq - \left(\dfrac{\mathcal{U}^1_{\delta}-\mathcal{U}^2_{\delta}}{r_2-r_1} \right)^{p-1}+  C\,,
\end{equation*}
which gives the upper bound in \eqref{lower_upper_bound}.

The lower bound in \eqref{lower_upper_bound} can be obtained by considering the function
$\overline v$ given by
\begin{equation*}
\overline{v}(x)=
\begin{cases}
-(\mathcal{U}^1_{\delta}-\mathcal{U}^2_\delta)\dfrac{H_0(x-\bar y)^{\frac{p-N}{p-1}}-\rho_2^{\frac{p-N}{p-1}}}{\rho_1^{\frac{p-N}{p-1}}-\rho_2^{\frac{p-N}{p-1}}}+\mathcal{U}^1_\delta & \text{if } 1<p<N\,, \\
 & \\
-\displaystyle(\mathcal{U}^1_{\delta}-\mathcal{U}^2_\delta)\dfrac{\ln (\rho_2^{-1} H_0(x-\bar y))}{\ln (\rho_2^{-1}\rho_1 )}+\mathcal{U}^1_\delta  & \text{if }  N=p\,
\end{cases}
\end{equation*}
in place of $\underline v$ and arguing as before, where now $\overline v$ serves as an upper barrier for $u_\delta$ (see also Proof of Step 1 of \cite{CiraoloSciammetta}).
\end{proof}

Let
 \begin{equation} \label{PsiNdelta}
\Psi_N(\delta)=
\begin{cases}
\delta^{-\frac{N+1}{2}+p} & p>\frac{N+1}{2} \,, \\
\left( \log\left(1/\delta\right)\right)^{-1}  &  p=\frac{N+1}{2} \,, \\
1 &  p<\frac{N+1}{2}\,.
\end{cases}
\end{equation}
We have the following lemma.

\begin{lemma}\label{lemma Idelta}
Let $w>0$ be fixed. Then
\begin{equation*}
- (1+\tau)(\mathcal{U}^1_{\delta}-\mathcal{U}^2_\delta)^{p-1}|\partial_{\xi_N} H_0(P_0)| ^{-1}\left(\frac{R_1+R_2}{2R_1 R_2}\right)^{-\frac{N-1}{2}}|\mathcal{Q}|^{-\frac{N-1}{2}} C_{N,p} \Psi^{-1}_N(\delta)(1+o(1))- C w^{N-1} \leq I_\delta(w),
\end{equation*}
and
\begin{equation*}
I_\delta(w) \leq -(1-\tau)(\mathcal{U}^1_{\delta}-\mathcal{U}^2_\delta)^{p-1}|\partial_{\xi_N} H_0(P_0)| ^{-1}\left(\frac{R_1+R_2}{2R_1 R_2}\right)^{-\frac{N-1}{2}}|\mathcal{Q}|^{-\frac{N-1}{2}} C_{N,p} \Psi^{-1}_N(\delta)(1+o(1))+ C w^{N-1},
\end{equation*}
as $\delta \to 0^+$.
\end{lemma}

\begin{proof}
Let $c=(1\pm\tau) \frac{R_1+R_2}{2R_1 R_2}$. From \eqref{lower_upper_bound} and \cite[Lemma AppendixB.1]{CiraoloSciammetta} and \cite[Lemma AppendixB.2]{CiraoloSciammetta}, we have
\begin{eqnarray*}
 - \left(\dfrac{\mathcal{U}^1_{\delta}-\mathcal{U}^2_{\delta}}{\delta +  c \QQ P^\perp \cdot P^\perp } \right)^{p-1}(1+ o(\delta^2 + |P-P_0|^2) )-  C  \leq H^{p-1}(\nabla u_\delta(P)) \nabla_\xi H(\nabla u_\delta (P)) \cdot\nu (P)\\
  \leq  - \left(\dfrac{\mathcal{U}^1_{\delta}-\mathcal{U}^2_{\delta}}{\delta +  c \QQ P^\perp \cdot P^\perp } \right)^{p-1}(1+ o(\delta^2 + |P-P_0|^2) )+  C \,,
\end{eqnarray*}
where $\mathcal{Q}$ is defined by \eqref{matriceQ} and $P^\perp$ is the projection of $P$ over the orthogonal to $P_0$.
By integrating over $ \mathcal{I}= \partial D_\delta^1 \cap \partial  \mathcal{N}_{\delta}(w)$, we obtain
\begin{eqnarray}\label{integraleIdelta}
 - \left(\mathcal{U}^1_{\delta}-\mathcal{U}^2_{\delta} \right)^{p-1}\int\limits_{\mathcal{I}}\dfrac{1}{\left(\delta +  c \QQ P^\perp \cdot P^\perp\right)^{p-1} }d\sigma (1+ o(1))-  Cw^{N-1}  \leq I_{\delta}(w)\\
  \leq  - \left(\mathcal{U}^1_{\delta}-\mathcal{U}^2_{\delta} \right)^{p-1}\int\limits_{\mathcal{I}}\dfrac{1}{\left(\delta +  c \QQ P^\perp \cdot P^\perp\right)^{p-1} }d\sigma (1+ o(1)))+  Cw^{N-1}.\nonumber
\end{eqnarray}
Hence, we have to understand the asymptotic behaviour of
\begin{equation} \label{integralone1}
\hat I = \int\limits_{\mathcal{I}} \dfrac{d\sigma}{\left(\delta +  c \QQ P^\perp \cdot P^\perp\right)^{p-1} }
\end{equation}
as $\delta \to 0$.

Since $\{x_N = 0\}$ is the orthogonal to $P^0$, we write $P^\perp$ as $P^\perp =(x',x_N)$, where $ x'=(x_1,\ldots,x_{N-1})$. The implicit function theorem guarantees that there exists a function $\phi : \{|\QQ^{1/2} x'|<w\} \to \mathbb{R}$ such that $H_0(x',\phi(x'))=R_1$, $\phi(0) = \delta$ and $(x',\phi(x')) \in \mathcal I$. Hence \eqref{integralone1} becomes
$$
\hat I =  \int\limits_{\{|\QQ^{1/2} x'|< w\}} \dfrac{ \sqrt{1+|\nabla_{x'} \phi(x')|^2}}{\left(\delta +  c \mathcal{Q} x' \cdot x'\right)^{p-1} } dx' \,.
$$
We recall that the point  $P=(x',\phi(x'))$ lies on a Wulff shape, which implies that
$$
1+|\nabla_{x'} \phi(x')|^2 =  \frac{1}{(\partial_{\xi_N} H_0(P_0))^2}(1+o(x')) \,,
$$
as $x'\to 0$ and, by a change of variables, we find
\begin{equation*}
\hat I =\left|\partial_{\xi_N} H_0(P_0)\right|^{-1} \delta^{\frac{N+1}{2}-p}\left(c\,\,|\mathcal{Q}|\right)^{-\frac{N-1}{2}} \int_{\big\{ |y|<\left(\frac{c}{\delta}\right)^{\frac{1}{2}}w \big\}} \dfrac{|y|^{N-2}}{(1 +  |y|^2)^{p-1} }dy \,.
\end{equation*}
Tedious by standard calculations show that 
$$
\lim_{\delta \to 0^+} \delta^{\frac{N+1}{2}-p}\Psi_N(\delta) \int_{\big\{ |y|<\left(\frac{c}{\delta}\right)^{\frac{1}{2}}w \big\}} \dfrac{|y|^{N-2}}{(1 +  |y|^2)^{p-1}} dy=K_{N,p}\,,
$$
where $\Psi_N(\delta)$ is given by \eqref{PsiNdelta} and $K_{N,p}$ is a constant which depends only on $N$ and $p$ and can be explicitly computed.
Hence
\begin{equation}\label{integraleIcappelleto}
 \widehat{I}= \left(c\,\,|\mathcal{Q}|\right)^{-\frac{N-1}{2}}K_{N,p} \left|\partial_{\xi_N} H_0(P_0)\right|^{-1}\Psi_N^{-1}(\delta)(1+o(1))
\end{equation}
as $\delta \to 0^+$.
From \eqref{integraleIdelta} and \eqref{integraleIcappelleto} we obtain the assertion.
\end{proof}

\noindent In the following Proposition we give upper and lower bounds on the difference of potential $\mathcal{U}_\delta^1 - \mathcal{U}_\delta^2$.

\begin{proposition}\label{proposizione_diffU1U2}
For any fixed $\tau \in (0,1/2)$ we have that
\begin{equation} \label{diff_U1U2}
(1-\tau) C_*  \Psi_N(\delta) (1 + o(1)) \leq \left(\mathcal{U}_\delta^1 - \mathcal{U}_\delta^2\right)^{p-1} \leq (1 + \tau) C_*\Psi_N(\delta) (1 + o(1)) \,,\nonumber
\end{equation}
where $C_*$ is given by \eqref{C*}.
\end{proposition}
\begin{proof}
  The proof follows straightforwardly from Lemma \ref{lemma Idelta} and Lemma \ref{lemmalimite}.
\end{proof}

We are ready to proof Theorem \ref{thm_main_1}.

\medskip

\begin{proof}[Proof of Theorem \ref{thm_main_1}] By choosing $w$ small enough and since the gradient stays uniformly bounded outside the neck of width $w$ (see Lemma \ref{lemma_bound_Omega_neck}), from Lemma \ref{lemma52} and Proposition \ref{proposizione_diffU1U2} we obtain the assertion.
\end{proof}

{\renewcommand{\addtocontents}[2]{}
	\section*{Acknowledgements}}

This work was supported by the project FOE 2014 ``Strategic Initiatives for the Environment and Security - SIES" of the Istituto Nazionale di Alta Matematica (INdAM) of Italy.
The authors have been partially supported by  the ``Gruppo Nazionale per l'Analisi Matematica, la Probabilit\`{a} e le loro Applicazioni (GNAMPA)'' of the ``Istituto Nazionale di Alta Matematica'' (INdAM).

\end{document}